\documentclass[a4paper,11pt]{article}
\usepackage[utf8]{inputenc}
\usepackage[T1]{fontenc}
\usepackage{graphicx}
\usepackage{amsmath,amsfonts,amssymb,amsthm}
\usepackage{mathtools}
\usepackage{mathrsfs}
\usepackage{bm}
\usepackage{color}
\usepackage{fullpage}

\definecolor{myred}{rgb}{0.77, 0.0, 0.1}
\definecolor{crimson}{rgb}{0.86, 0.08, 0.24}
\definecolor{awesome}{rgb}{1.0, 0.13, 0.32}
\definecolor{newgreen}{rgb}{0.0,0.6,0.0}
\definecolor{malachite}{rgb}{0.04, 0.85, 0.32}
\definecolor{pastelgreen}{rgb}{0.47, 0.87, 0.47}
\definecolor{myturq}{rgb}{0.1, 0.7, 0.7}

\usepackage{hyperref}
\hypersetup{
  colorlinks,
  linkcolor=blue, 
  citecolor=blue, 
  linktoc=all 
}

\usepackage{relsize}

\renewcommand{\leq}{\leqslant}
\renewcommand{\geq}{\geqslant}

\newcommand{\mgeq}{\succcurlyeq}

\newcommand{\di}{\mathrm{d}}

\newcommand{\eps}{\varepsilon}

\newcommand{\argmin}{\mathop{\arg\min}}
\newcommand{\argmax}{\mathop{\arg\max}}

\newcommand{\wt}{\widetilde}
\newcommand{\wh}{\widehat}
\newcommand{\pp}{\, : \; }

\newcommand{\R}{\mathbb R}

\usepackage{xspace}
\newcommand{\ie}{\textit{i.e.}\@\xspace} 
\newcommand{\eg}{e.g.\@\xspace}
\newcommand{\iid}{i.i.d.\@\xspace}

\newcommand{\id}{I_d}

\newcommand{\tr}{\mathrm{tr}}
\newcommand{\Span}{\mathrm{span}}

\newcommand{\opnorm}[1]{\| {#1} \|_{\mathrm{op}}}
\newcommand{\dist}{\mathrm{dist}}

\newcommand{\E}{\mathbb E}
\renewcommand{\P}{\mathbb P}

\newcommand{\cond}{|}

\newcommand{\indic}[1]{\bm 1 ( #1 )}
\newcommand{\kl}{\mathrm{KL}}
\newcommand{\kll}[2]{\kl ({#1}, {#2})}

\newcommand{\gaussdist}{\mathcal{N}}

\newcommand{\cvdist}{\stackrel{(\mathrm{d})}{\to}}

\newcommand{\G}{\mathcal{G}}

\newcommand{\F}{\mathcal{F}}
\newcommand{\X}{\mathcal{X}}
\newcommand{\Y}{\mathcal{Y}}
\newcommand{\Zs}{\mathcal{Z}}

\newcommand{\Fh}{\widehat{\mathcal{F}}}

\renewcommand{\H}{\mathcal{H}}

\newcommand{\excessrisk}{\mathcal{E}}
\newcommand{\predspace}{\widehat{\Y}}
\newcommand{\pred}{\widehat{y}}

\newcommand{\deff}{d_{\mathrm{eff}}}
\newcommand{\df}[2]{\mathsf{df}_{#1} (#2)}
\newcommand{\dflambda}[1]{\df{\lambda}{#1}}

\newtheorem{proposition}{Proposition}
\newtheorem{theorem}{Theorem}
\newtheorem{lemma}{Lemma}
\newtheorem{corollary}{Corollary}

\theoremstyle{definition}

\newtheorem{assumption}{Assumption}

\theoremstyle{remark}
\newtheorem{remark}{Remark}

\title{An improper estimator with optimal excess risk in misspecified density estimation and logistic regression}

\author{Jaouad Mourtada\footnote{CREST, ENSAE, Institut Polytechnique de Paris, France.
    The research leading to this work was carried while the first author was a PhD student at CMAP, École polytechnique, France.} 
  \qquad St\'ephane Ga\"iffas\footnote{LPSM, UMR 8001, Universit\'e de Paris, and DMA, UMR 8553, Ecole normale sup\'erieure, Paris, France}
}
\date{}

\usepackage[all]{xy}

\begin{document}

\maketitle

\begin{abstract}
  We introduce a procedure for conditional density estimation under logarithmic loss, which we call SMP (Sample Minmax Predictor).
  This estimator minimizes a new general excess risk bound for statistical learning.
  On standard examples, this bound scales as $d/n$ with $d$ the model dimension and $n$ the sample size, and critically remains valid under model misspecification.
  Being an improper (out-of-model) procedure, SMP improves over within-model estimators such as the maximum likelihood estimator, whose excess risk degrades under misspecification.
  Compared to approaches reducing to the sequential problem, our bounds remove suboptimal $\log n$ factors and can handle unbounded classes.
  For the Gaussian linear model, the predictions and risk bound of SMP are governed by leverage scores of covariates, nearly matching the optimal risk in the well-specified case without conditions on the noise variance or approximation error of the linear model.
  For logistic regression, SMP provides a non-Bayesian approach to calibration of probabilistic predictions relying on virtual samples, and can be computed by solving two logistic regressions.
  It achieves a non-asymptotic excess risk of $O ( (d + B^2R^2)/n )$, where $R$ bounds the norm of features and $B$ that of the comparison parameter; by contrast, no within-model estimator can achieve better rate than $\min( {B R}/{\sqrt{n}}, {e^{BR}}/{n} )$ in general~\cite{hazan2014logistic}.
  This provides a more practical alternative to Bayesian approaches, which require approximate posterior sampling, thereby partly addressing a question raised by Foster et al.~\cite{foster2018logistic}.
 
\end{abstract}

\section{Introduction}
\label{sec:introduction}

Consider the standard problem of density estimation: given an \iid sample $Z_1, \dots, Z_n$ from an unknown distribution $P$ on some measurable space $\Zs$, the goal is to produce a good approximation $\wh P_n$ of $P$.
One way to measure the quality of an estimate $\wh P_n$ is through its predictive risk: given a base measure $\mu$ on $\Zs$, the risk of a density $g$ on $\Zs$ with respect to $\mu$ is given by
\begin{equation}
  \label{eq:def-pred-risk}
  R (g) = \E [ \ell (g, Z) ]
  \, ,
  \quad \mbox{ where } \quad
  \ell (g, z) = - \log g (z)
\end{equation}
for $z \in \Zs$ and where $Z$ is a random variable with distribution $P$.
Letting $\G$ denote the set of all probability densities on $\Zs$ with respect to $\mu$, the loss function $\ell: \G \times \Zs \to \R$ defined by~\eqref{eq:def-pred-risk}, called \emph{logarithmic} (or \emph{negative log-likelihood}, \emph{entropy} or \emph{logistic}) loss, measures the error of the density $g \in \G$ (which can be interpreted as a probabilistic prediction of the outcome) given outcome $z \in \Zs$.
This loss function is standard in the information theory literature, due to its link with coding~\cite{cover2006elements}.
The risk of a density $g$ can be interpreted in relation to
the joint probability assigned by $g$ to a large \iid test sample $Z_1', \dots, Z_m'$ from $P$: by the law of large numbers, as $m$ tends to infinity, almost surely
\begin{equation*}
  \prod_{j=1}^m g (Z_j')
  = \exp \Big( - \sum_{j=1}^m \ell (g, Z_j') \Big)
  = \exp \Big( - m [R (g) + o(1) ] \Big)
  \, .
\end{equation*}
In addition, assume that $P$ of $Z$ has a density $p \in \G$; we then have, for every $g \in \G$,
\begin{equation*}
  R (g) - R (p)
  = \E \Big[ \log \Big( \frac{p (Z)}{g (Z)} \Big) \Big]
  = \int_{\Zs} \log \Big( \frac{p}{g} \Big) p \, \di \mu
  = \kll{p \cdot \mu}{g \cdot \mu}
  \geq 0
  \, ,
\end{equation*}
where $\kll{P}{Q} := \int_\Zs \log \big( \frac{\di P}{\di Q} \big) \di P$ denotes the \emph{Kullback-Leibler divergence} (or \emph{relative entropy})  between distributions $P$ and $Q$.
In particular, the risk is minimized by the true density $p$ (if it exists), and prediction under logarithmic loss is equivalent to density estimation under Kullback-Leibler risk.

Our aim is to find \emph{estimators}, which associate to any sample $Z_1, \dots, Z_n$ a density $\wh g_n \in \G$, whose risk is controlled in some general setting.
While it is typically impossible to obtain finite-sample guarantees without any assumption on the underlying distribution $P$ (see \eg \cite[Section~7.1]{devroye1996ptpr}),
oftentimes one expects this distribution to possess some structure.
In such cases, it is natural to introduce inductive bias in the procedure;
one standard way to do so is to select a suitable class of densities
$\F \subset \G$ (often called a \emph{statistical model}) that is susceptible to capture at least part of the structure of $P$, and thus provide a non-trivial approximation of it.

A classical approach is then to assume that the model $\F$ is \emph{well-specified}, in the sense that it contains the true density $p$.
In this case, the problem of estimating $P$ falls within the classical
framework of parametric statistics~\cite{ibragimov1981estimation,vandervaart1998asymptotic,lehmann1998tpe}.
This theory provides strong support for the maximum likelihood estimator (MLE), which arises as an asymptotically optimal estimator for regular models as the sample size $n$ grows \cite{hajek1972local,lecam1986asymptotic,ibragimov1981estimation}.
The same problem can also be treated for a fixed sample size, through the lens of statistical decision theory \cite{wald1949decision,lehmann1998tpe}, which
emphasizes optimal estimators in the average (Bayesian) and minimax senses.
Generally speaking, these approaches offer precise descriptions of achievable rates of convergence (up to correct leading constants) and of \emph{efficient} estimators that make the best use of available data.
A major limitation of this approach, however, is that these results rely on the unrealistic assumption that the true distribution belongs to the selected model.
Such an assumption is generally unlikely to hold, since the model usually involves a simplified representation of the phenomenon under study: it comes from a choice of the statistician, who has no control over the true distribution.

A more realistic situation occurs when the underlying model captures some aspects of the true distribution, such as its most salient properties, but not all of them.
In other words, the statistical model provides some non-trivial approximation of the true distribution, and is thus ``wrong but useful''.
In such a case, a meaningful objective is to approximate the true distribution (namely, to predict its realizations) almost as well as the best distribution in the model.
This task can naturally be cast in the framework of Statistical Learning Theory \cite{vapnik1998learning}, where one constrains the comparison class $\F$ while making few modeling assumptions about the true distribution.
Given a class $\F$ of densities, the performance of an estimator $\wh g_n$ is evaluated in terms of its \emph{excess risk} with respect to the class $\F$, namely
\begin{equation*}
  \excessrisk (\wh g_n) := R (\wh g_n) - \inf_{f \in \F} R (f)\, .
\end{equation*}
We say that the estimator $\wh g_n$ is \emph{proper} (or a \emph{plug-in} estimator) when it takes value inside the class $\F$, otherwise $\wh g_n$ will be referred to as an \emph{improper} procedure.
Below, we discuss two established approaches to this problem.

\paragraph{Maximum Likelihood Estimation.}

Arguably the simplest and most standard procedure is the \emph{Maximum Likelihood Estimator} (MLE), or \emph{Empirical Risk Minimizer} (ERM) with logarithmic loss, given by
\begin{equation}
  \label{eq:def-mle}
  \wh f_n := \argmin_{f \in \F} \frac{1}{n} \sum_{i=1}^n \ell (f, Z_i)
  = \argmax_{f \in \F} \prod_{i=1}^n f (Z_i)
  \, .
\end{equation}
Assume now that $\F = \{ f_{\theta} : \theta \in \Theta \}$ is some parametric model indexed by an open subset $\Theta \subset \R^d$, such that the density $f_\theta (z)$ depends smoothly on $\theta$, and denote $\wh f_n = f_{\wh \theta_n}$ the MLE.
First, consider the well-specified case where the true distribution $P$ belongs to the model, say $P = f_{\theta^*} \cdot \mu$, and denote $I (\theta^*) := \E [ - \nabla^2 \log f_\theta (Z) ]|_{\theta = \theta^*}$ the Fisher information matrix, assumed invertible.
Then, under standard regularity and moment conditions \cite{vandervaart1998asymptotic,ibragimov1981estimation}, we have as $n \to \infty$,
\begin{equation*}
  \sqrt{n} (\wh \theta_n - \theta^*) \cvdist \gaussdist (0, I (\theta^*)^{-1})
  \quad
  \mbox{ while }
  \quad
  \excessrisk (f_\theta) = \frac{1}{2} \| \theta - \theta^* \|_{I(\theta^*)}^2 + o (\| \theta - \theta^* \|^2)
  \, ,
\end{equation*}
where we denote $\| u \|_{A} := \langle A u, u\rangle^{1/2}$ for any $u \in \R^d$ and symmetric positive matrix $A$.
This implies that $2 n \excessrisk (f_{\wh \theta_n})$ converges in distribution to a $\chi^2_d$ distribution; hence, under suitable domination conditions, the asymptotic excess risk of the MLE satisfies $\E [ \excessrisk (\wh f_n) ] = d/(2 n) + o(n^{-1})$.
This asymptotic performance turns out to be unimprovable in the well-specified case: for instance, MLE is locally asymptotically minimax optimal~\cite{hajek1972local,lecam2000asymptotics}.

In contrast to its optimality in the well-specified case, the performance of MLE can degrade in the general misspecified case, where it depends on the true distribution $P$.
Indeed, let $\theta^* = \argmin_{\theta \in \Theta} R (f_{\theta^*})$ be the optimal parameter, and $G = \E [ \nabla \ell (f_{\theta}, Z) \nabla \ell (f_\theta, Z)^\top ]|_{\theta = \theta^*}$, $H = \E [ \nabla^2 \ell (f_{\theta}, Z) ]|_{\theta = \theta^*}$;
when $P$ belongs to the model, $G = H = I (\theta^*)$, but in general those matrices are distinct.
In this case, under suitable conditions, it follows from general results on the asymptotic behavior of $M$-estimators \cite{vandervaart1998asymptotic,white1982maximum} that
\begin{equation*}
  \sqrt{n} (\wh \theta_n - \theta^*) \cvdist \gaussdist (0, H^{-1} G H^{-1})
  \quad \mbox{ and } \quad
  \excessrisk (f_{\theta}) = \frac{1}{2} \| \theta - \theta^* \|_H^2 + o(\| \theta - \theta^* \|^2)
  \, .
\end{equation*}
Again under suitable domination conditions, this implies that, as $n \to \infty$,
\begin{equation}
  \label{eq:asymp-risk-mle-misspecified}
  \E [ \excessrisk (\wh f_n) ]
  = \frac{\tr (H^{-1/2} G H^{-1/2})}{2 n} + o \Big( \frac{1}{n} \Big)
  = \frac{\deff}{2 n} + o \Big( \frac{1}{n} \Big)
  \, ;
\end{equation}
here, the constant $\deff := \tr (H^{-1/2} G H^{-1/2})$ depends on the distribution $P$, and can typically be arbitrarily large, as will be seen below in the case of logistic regression.
In fact, degradation under model misspecification is not specific 
to MLE, and is typically a limitation shared by any \emph{proper} (or \emph{plug-in})
estimator that returns a distribution within the class $\F$, such as penalized MLE.
Finally, let us stress that, while we adopted a simplified asymptotic viewpoint in this discussion for the sake of clarity, our focus will be on finite-sample bounds, with an explicit dependence on problem-dependent parameters.

\paragraph{Sequential prediction and online-to-offline conversion.}

In contrast, distribution-free excess risk bounds have been obtained in the literature~\cite{barron1987bayes,catoni2004statistical,yang2000mixing,juditsky2008mirror,audibert2009fastrates} through a reduction to the comparatively much
better understood setting of sequential prediction under logarithmic loss~\cite{merhav1998universal,cesabianchi2006plg,shtarkov1987universal,grunwald2007mdl}.
In this problem, which is connected to coding~\cite{cover2006elements} and the minimum description length (MDL) principle~\cite{rissanen1985mdl,grunwald2007mdl},
one seeks to control \emph{cumulative} criteria such as the cumulative excess risk, or the regret
\begin{equation*}
  \sum_{i=1}^n \ell (\wh g_{i-1}, Z_i) - \inf_{f \in \F} \sum_{i=1}^n \ell (f, Z_i)
\end{equation*}
over all sequences $Z_1, \dots, Z_n \in \Zs$, where $\wh g_{i-1}$ is selected based on
$Z_1, \dots, Z_{i-1}$.
The control of such cumulative quantities is significantly simplified by the observation that
\begin{equation*}
  \sum_{i=1}^n \ell (\wh g_{i-1}, Z_i) - \inf_{f \in \F} \sum_{i=1}^n \ell (f, Z_i)
  = - \log \bigg( \frac{\prod_{i=1}^n \wh g_{i-1} (Z_i)}{\sup_{f \in \F} \prod_{i=1}^n f (Z_i)} \bigg) \, ,
\end{equation*}
where the ratio inside the logarithm can be interpreted as a ratio of joint densities over $Z_1, \dots, Z_n$.
This enables one to determine the minimax regret~\cite{shtarkov1987universal}, as well as to control the regret of specific sequential prediction strategies $\wh g_0, \dots, \wh g_{n-1}$.
Among those, arguably the most standard are Bayesian mixture strategies~\cite{vovk1998mixability,littlestone1994weighted,merhav1998universal,cesabianchi2006plg} with near-optimal guarantees~\cite{clarke1994jeffreys,xie2000minimax,merhav1998universal,cesabianchi2006plg}, where given a prior distribution $\pi$ on the parameter space $\Theta$, $\wh g_{i}$ is the \emph{Bayesian predictive posterior}:
\begin{equation}
  \label{eq:bayes-pred-posterior}
  \wh g_i (z)
  = \frac{\int_\Theta f_\theta (Z_1) \cdots f_\theta (Z_i) f_\theta (z) \pi (\di \theta)}{\int_\Theta f_\theta (Z_1) \cdots f_\theta (Z_i) \pi (\di \theta)}
  = \int_\Theta f_\theta (z) \, \pi (\di \theta \cond Z_1, \dots, Z_i)
  \, .
\end{equation}
For smooth, bounded parametric families of dimension $d$, the minimax cumulative excess risk and regret are known to scale as $(d \log n) /2 + C (\F)$ for some constant $C (\F)$ depending on the model, see~\cite{clarke1994jeffreys,merhav1998universal}.
Note that regret bounds hold for any sequence, and in particular do not require the sequence of observations to be sampled from a distribution in the model.
A generic procedure called \emph{online to batch conversion}~\cite{littlestone1989online2batch,cesabianchi2004online_to_batch} enables
one to convert any guarantee on the cumulative excess risk into one on the \emph{non-cumulative} excess risk for the average of the successive densities output by the sequential procedure, namely
\begin{equation}
  \label{eq:progressive-mixture}
  \bar g_n = \frac{1}{n+1} \sum_{i=0}^{n} \wh g_i
  \, .
\end{equation}
When applied to Bayes mixture rules, this yields the so-called \emph{progressive mixture} or \emph{mirror averaging} procedure~\cite{yang1999information,catoni1997mixture,catoni2004statistical,juditsky2008mirror,audibert2009fastrates}, with excess risk bounded by $O ( (d \log n) / n + C (\F) / n )$.

While appropriate for sequential prediction, this approach is not
fully satisfactory in the statistical learning setting considered here, for the following reasons.
First, the obtained $O (d \log n / n)$ rate features a suboptimal $\log n$ factor, when compared to the $O (d/n)$ rate of MLE in the well-specified case; this highlights the inefficiency of the averaged estimator $\bar g_n$, 
which mixes estimators $\wh g_i$ computed with only a fraction of the sample.
Obtaining bounds of $O (d/n)$ for the excess risk was posed as an open problem (first question in \cite{grunwald2011open-individual}).
Second, the minimax regret (and in particular the model-dependent constant $C (\F)$) is typically \emph{infinite}~\cite{shtarkov1987universal,clarke1994jeffreys,rissanen1996fisher,grunwald2007mdl} for unbounded ``infinite-volume''
classes $\F$ including Gaussian models, so that no uniform guarantee can be obtained over such classes through regret minimization and online-to-offline conversion, reflecting the poor localization of such bounds.
These first two limitations are shared by any approach reducing to the sequential problem, which takes into account early rounds where few observations are available.
A third limitation lies in the computational requirements of such procedures: in particular, Bayesian mixture approaches involve --- absent a conjugate prior allowing exact computations --- approximate posterior computations, which are often
significantly more expensive than maximum likelihood optimization, inhibiting practical use of such methods.

\subsection{Our contributions}
\label{sec:intro-contribution}

Let us now summarize our main contributions.
Note that, while the previous discussion dealt with density estimation, most of this work in fact deals with \emph{conditional} density estimation, where one seeks to estimate the conditional distribution of a response $Y$ to an input variable $X$, under logarithmic loss $\ell (f, (X, Y)) = - \log f (Y \cond X)$ (see Section~\ref{sec:excess-risk-log}).

\paragraph{SMP: a procedure for conditional density estimation.}

We introduce a general procedure for predictive density estimation under entropy risk.
This estimator, which we call \emph{Sample Minmax Predictor} (SMP), is obtained by minimizing a new general excess risk bound for supervised statistical learning (Theorem~\ref{thm:excess-risk-stat}), and in particular conditional density estimation (Theorem~\ref{thm:excess-risk-log}).
In short, SMP is the solution of some minmax problem obtained by considering \emph{virtual samples}.
SMP satisfies an excess risk bound valid under model misspecification, and unlike previous approaches does not rely on a reduction to the sequential problem, thereby improving rates for parametric classes from $O (d \log n / n)$ to $O (d/n)$ for our considered models, addressing a question raised by~\cite{grunwald2011open-individual} in these cases (``individual risk'' in this paper refers to the excess risk).

\paragraph{SMP for the Gaussian linear model.}

We apply SMP to the \emph{Gaussian linear model} $\F = \{ f_\theta (\cdot \cond x) = \gaussdist (\langle \theta, x\rangle, \sigma^2) : \theta \in \R^d \}$ for some $\sigma^2 > 0$, a classical conditional model for a scalar response $y \in \R$ to covariates $x \in \R^d$.
SMP then smoothes predictions in terms of \emph{leverage scores}, and for every distribution of covariates, its expected excess risk in the general misspecified case is at most twice the minimax excess risk in the \emph{well-specified} case, but without any condition on the approximation error of the linear model or noise variance (Theorem~\ref{thm:smp-gaussian-linear}).
This yields an excess risk bound of $d/n + O ((d/n)^2)$ over the class $\F$ under some regularity assumptions on covariates (Corollary~\ref{cor:smp-linear-nonasymptotic}); such a guarantee cannot be obtained for a within-model estimator, or through a regret minimization approach.

We also consider a Ridge-regularized variant of SMP, and study its performance on
balls of the form $\F_B = \{ f_\theta : \| \theta \| \leq B \}$ for $B > 0$.
For covariates $X$ bounded by $R > 0$, we establish two guarantees: a ``finite-dimensional'' bound of $O (d \log (BR / \sqrt{d}) / n)$ (Proposition~\ref{prop:smp-ridge-log-norm}), removing an extra $\log n$ term from results of \cite{kakade2005online} in the sequential case, and a dimension-free ``nonparametric'' bound (Theorem~\ref{thm:smp-linear-ridge}), where explicit dependence on $d$ is replaced by a dependence on the covariance structure of covariates, matching well-specified minimax rates over such balls in infinite dimension~\cite{caponnetto2007optimal}.

\paragraph{SMP for logistic regression.}

We then turn to logistic regression, arguably the most standard model for a binary response $y \in \{ - 1, 1\}$ to covariates $x \in \R^d$, given by $\F = \{ f_{\theta} (1 \cond x) = \sigma (\langle \theta, x\rangle) : \theta \in \R^d \}$, where $\sigma (u) = e^u / (1 + e^u)$.
In this case, SMP admits a simple form, and its prediction can be computed by solving two logistic regressions.
Assuming that $\| X \| \leq R$, we show that a Ridge-penalized variant of SMP achieves excess risk $O ( (d + B^2 R^2) / n )$ with respect to the ball $\F_B = \{ f_\theta : \| \theta \| \leq B \}$ for all $B > 0$ (Corollary~\ref{cor:logistic-risk-smp-finitedim}), together with dimension-free bounds (Theorem~\ref{thm:logistic-ridge-smp}).
In contrast, results of~\cite{hazan2014logistic} show that no within-model estimator can achieve better rate than $\min (BR/\sqrt{n}, e^{BR}/n)$ without further assumptions.
Compared to approaches obtaining fast rates through Bayesian mixtures~\cite{kakade2005online,foster2018logistic}, computation of SMP replaces posterior sampling by optimization.
SMP thus provides a natural non-Bayesian approach to uncertainty quantification and calibration of probabilistic estimates, relying on virtual samples.

\subsection{Related work}
\label{sec:related-work}

\paragraph{Well-specified density estimation.}

There is a rich statistical literature on predictive density estimation under entropy risk in the well-specified case (where the true distribution is assumed to belong to the model), see~\cite{harris1989predictive,komaki1996predictive,hartigan1998maximum,aslan2006asymptotically,liang2004exact,george2006improved,sweeting2006nonsubjective,brown2008admissible} and references therein.
First, as mentioned above, MLE is known to be asymptotically normal and efficient~\cite{vandervaart1998asymptotic,ibragimov1981estimation,lecam2000asymptotics} in this case; its asymptotic optimality can be formalized precisely by H\'ajek's local asymptotic minimax theorem~\cite{hajek1972local,lecam2000asymptotics}.
Beyond this optimality result, a number of refinements have been explored: improvement of Bayes predictive distributions over the MLE for finite samples~\cite{aitchison1975goodness}, higher-order risk asymptotics~\cite{hartigan1998maximum,ghosh1994higher,aslan2006asymptotically} and second-order minimax procedures~\cite{aslan2006asymptotically}, exact minimax procedures for location and scale families~\cite{liang2004exact}, as well as admissibility and shrinkage for the Gaussian model~\cite{brown2008admissible}.
While related to this line of work, our approach differs from it by relaxing the (restrictive) assumption that the distribution of interest belongs to the specified model; another difference with some of the aforementioned references is our non-asymptotic focus.

\paragraph{Non-asymptotic analyses of estimators under misspecification.}

The asymptotic behavior of MLE (including consistency and asymptotic normality) in the misspecified case is also well-understood~\cite{white1982maximum,vandervaart1998asymptotic}.
Beyond the asymptotic setting, non-asymptotic analyses of MLE and related procedures
have been carried~\cite{vandegeer1999empirical,birge1993rates,birge1998contrast,yang1998asymptotic,wong1995probability,spokoiny2012parametric}, by using techniques from empirical process theory~\cite{vandervaart1996weak,talagrand2014upper,massart2007concentration,boucheron2013concentration}.
In addition to these classical references, we mention two approaches that circumvent in different ways reliance on the machinery of empirical process theory.
First,~\cite{zhang2006entropy} relies on information-theoretic inequalities to analyze Bayesian and penalized estimators; this approach is extended by~\cite{grunwald2019tight}, who obtain bounds in terms of refined complexity measures.
Second,~\cite{ostrovskii2018finite} developed an analysis relying on self-concordance,
which applies in particular to logistic regression.
Overall, this literature differs from ours in that it studies estimators such as (penalized) MLE, which inevitably degrade for some misspecified distributions.

\paragraph{Sequential prediction.}

As mentioned previously, the sequential variant of prediction under logarithmic loss is well-studied~\cite{shtarkov1987universal,clarke1994jeffreys,merhav1998universal,vovk1998mixability,cesabianchi2006plg,grunwald2007mdl}.
These guarantees on cumulative criteria have been transported to the individual excess risk considered here~\cite{barron1987bayes,catoni2004statistical,yang1999information,juditsky2008mirror,audibert2009fastrates}.
As it turns out, this online-to-offline conversion is the main approach to obtaining distribution-free excess risk guarantees.
As mentioned above, reduction to the sequential case is suboptimal, in that it leads to extra logarithmic factors in the rate and cannot provide uniform guarantees over unbounded models.
Our general guarantee for SMP provides a more ``localized'' risk bound adapted to such situations.

\paragraph{Stability.}

Our general bound on the excess risk is related to the approach in terms of stability of the loss of the predictor under sample changes~\cite{bousquet2002stability,rakhlin2005stability,shalev2010learnability,koren2015expconcave}, in particular in its use of exchangeability.
While close in spirit, our bounds involve a different quantity; the difference between the two is particularly apparent in the context of logistic regression, where it enables us to remove some exponential constants.

\paragraph{Logistic regression.}

An important motivation for this work was recent progress and questions on logistic regression, arguably the most common model for conditional density estimation with binary response~\cite{berkson1944logistic,mccullagh1989glm,vandervaart1998asymptotic}.
Under boundedness assumptions, it can be seen as a special convex and Lipschitz stochastic optimization problem, for which slow rates of convergence are available~\cite{zinkevitch2003onlineconvex,nemirovski2009robust,bubeck2015convex}.
In addition, logistic regression is also an exp-concave problem, which enables fast rates~\cite{hazan2007logarithmic,koren2015expconcave,mehta2017expconcave_statistical}, but with an exponential dependence on the domain radius.
It is shown by~\cite{hazan2014logistic} that such rates are unimprovable without further assumptions.
To obtain improved results, one thread of work proceeds under additional assumptions, and performs a refined analysis using (generalized) \emph{self-concordance} of the logistic loss~\cite{bach2010logistic,bach2014logistic,bach2013nonstrongly,ostrovskii2018finite,marteau2019globally}; this leads to distribution-dependent guarantees which improve for favorable distributions, but exhibit exponential dependence in the worst case.
Another approach consists in using out-of-model procedures, for which the lower bound of~\cite{hazan2014logistic} does not apply.
By using Bayes mixtures strategies and reducing to the sequential problem, \cite{kakade2005online,foster2018logistic} establish fast risk rates without exponential dependence on the norm, bypassing the previous lower bound;
the question of finding a practical procedure enjoying such guarantees without expensive posterior sampling is left open in~\cite{foster2018logistic}.
Our work is cast in the same setting under weak distributional assumptions, and provides a practical approach with  fast rates guarantees in this case.
We note that our analysis of SMP also relies on self-concordance, though it applies to a different estimator.

\paragraph{Min-max procedures for linear regression.}
SMP provides a simple alternative to Bayes mixtures for logistic regression.
Interestingly, similar procedures have been introduced for linear regression with square loss, namely the Vovk-Azoury-Warmuth (VAW) algorithm~\cite{vovk2001competitive,azoury2001relative} for sequential prediction, and the Forster-Warmuth estimator~\cite{forster2002relative} for statistical learning.
Our estimator can be seen as a counterpart for density estimation and logistic regression to these linear regression procedures in several ways.
First, they simplify over a general aggregation approach that would ``mix'' over a continuous class. 
In addition, the predictions from these methods involve a leverage correction, like SMP for Gaussian linear density estimation
(though the correction concerns the variance rather than the mean for SMP).
Finally, while \cite{vovk2001competitive} derives the VAW procedure from aggregation and mixability considerations,
the approach of \cite{forster2002relative,azoury2001relative} also relies on a min-max loss difference.

\subsection{Outline and notation}
\label{sec:outline-notations}

This paper is organized as follows.
In Section~\ref{sec:general-excess-risk}, we introduce the setting and state a general excess risk bound for supervised learning (Theorem~\ref{thm:excess-risk-stat}) and its instantiation to conditional density estimation (Theorem~\ref{thm:excess-risk-log}), minimized by SMP, which will be used throughout.
Section~\ref{sec:density-estimation} provides direct consequences of the previous bounds in the context of (unconditional) density estimation with multinomial and Gaussian models.
In Section~\ref{sec:cond-gauss}, we study SMP and its guarantees for conditional density estimation with the Gaussian linear model.
We finally turn to logistic regression in Section~\ref{sec:logistic}.
The proofs are gathered in Section~\ref{sec:proofs}, while Section~\ref{sec:conclusion} concludes.

\paragraph{Notation.}

Throughout this text, we denote $\langle x, y\rangle := x^\top y$ the canonical scalar product of $x, y \in \R^d$, and $\| x \| := \langle x, x\rangle^{1/2}$ the associated Euclidean norm.
Likewise, for any symmetric positive semi-definite $d \times d$ matrix $\Sigma$, we let $\langle x, y\rangle := \langle \Sigma x, y\rangle$ and $\| x \|_\Sigma = \langle x, x\rangle_\Sigma^{1/2}$.
We denote by $\dist (x, A) = \inf_{y \in A} \| x - y\|$ the distance of a point $x \in \R^d$ to a subset $A \subset \R^d$.

\section{General excess risk bounds}
\label{sec:general-excess-risk}

\subsection{A general excess risk bound for statistical learning}
\label{sec:excess-risk-statistical-learning}

In this section, we let $\X, \Y, \predspace$ be three measurable spaces, corresponding respectively to the feature, label and prediction spaces, and  let $\ell : \predspace \times \Y \to \R$ be a loss function.
Denote by $\Fh$ the space of all measurable functions $\X \to \predspace$ (also called predictors), and let $\F \subset \Fh$ be a class of predictors.
We also consider a penalization function
$\phi : \F \to \R$.
Denote $\Zs = \X \times \Y$ and let
\begin{equation*}
  \ell_\phi (f, z) = \ell (f(x), y) + \phi(f)
\end{equation*}
for any $z=(x,y) \in \Zs$ and $f \in \F$.
When no penalization is used ($\phi \equiv 0$) we simply write $\ell =\ell_0$.
Let $P$ be some probability distribution on $\Zs = \X \times \Y$.
The quality of a predictor $g \in \Fh$ is measured through its \emph{risk} 
\begin{equation}
  \label{eq:def-risk-statlearn}
  R (g) = \E [\ell (g, Z)] = \E [ \ell (g(X), Y) ]
\end{equation}
where $Z = (X,Y) \sim P$, whenever this expectation is well-defined and belongs to $\R \cup \{ + \infty \}$, which we assume from now on.
Also, define the \emph{excess risk} (with respect to $\F$) of $g$ as
\begin{equation}
  \label{eq:def-excess-risk-statlearn}
  \excessrisk (g) = R (g) - \inf_{f \in \F} R (f)
  \, .
\end{equation}
We define similarly $R_\phi(f) = \E [\ell_\phi (f, Z)]$ for $f \in \F$ 
and $\excessrisk_\phi (g) = R (g) - \inf_{f \in \F} R_\phi (f)$.

In this setting, the distribution $P$ is unknown, and we will avoid making strong assumptions on it. 
The aim is to produce, given an \iid sample
$Z_1^n = (Z_1, \dots, Z_n)$ from $P$, a predictor
 $\wh g_n: \X \to \wh \Y$ whose expected excess risk $\E [\excessrisk (\wh g_n)]$ (where the expectation holds over the random sample) is small.
In other words, $\wh g_n$ should predict almost as well as the best element in $\F$, up to a controlled small additional term.
Given a sample $Z_1^n = (Z_1, \dots, Z_n)$, we denote
\begin{equation}
  \label{eq:def-erm}
  \wh f_{\phi, n} \in \argmin_{f \in \F} \sum_{i=1}^n \ell_\phi (f, Z_i)
\end{equation}
a (penalized) \emph{empirical risk minimizer} (ERM);
when $\phi \equiv 0$, we simply denote the ERM as $\wh f_n$.
Throughout this paper, we assume to simplify that this minimum is attained.
This holds in virtually all the examples considered below; in addition, the arguments naturally extend to approximate minimizers.
By convention, all minimizers of the empirical risk will be chosen symmetrically in the sample points $Z_1, \dots, Z_n$.
We also introduce
\begin{equation}
  \label{eq:def-erm-next}
  \wh f_{\phi, n}^{z}  
  := \argmin_{f \in \F} \bigg\{ \sum_{i=1}^n \ell_\phi (f, Z_i) + \ell_\phi (f, z) \bigg\}
\end{equation}
for any $z \in \Zs$.
Theorem~\ref{thm:excess-risk-stat} below introduces a new bound on the excess risk of any prediction rule, together with a predictor that minimizes it.
It holds for a general loss $\ell$, but in the following sections we apply it to the logarithmic loss only, for which the predictor can be made explicit.

\begin{theorem}[Main excess risk bound and Sample Minmax Predictor]
  \label{thm:excess-risk-stat}
  For any predictor $\wh g_n$ depending on $Z_1^n$\textup, we have
  \begin{equation}
    \label{eq:main-excess-risk-stat}
    \E \big[ \excessrisk_\phi (\wh g_n) \big]
    \leq \E_{Z_1^n, X} \Big[ \sup_{y \in \Y} \Big\{ \ell (\wh g_n (X), y) - \ell_\phi (\wh f_{\phi, n}^{(X,y)} (X), y) \Big\} \Big]
  \end{equation}
  where $\wh f_{\phi, n}^{z}$ is defined by~\eqref{eq:def-erm-next} for $z \in \Zs$ and $Z = (X,Y) \sim P$ is independent of $Z_1^n$.
  In addition\textup, the right-hand side of~\eqref{eq:main-excess-risk-stat} is minimized by the predictor
  \begin{equation}
    \label{eq:pred-rule-stat}
    \wt f_{\phi, n} (x) = \argmin_{\pred \in \predspace} \sup_{y \in \Y} 
    \Big \{ \ell (\pred, y) - \ell_\phi (\wh f_{\phi, n}^{(x,y)} (x), y) \Big\}
    \, ,
  \end{equation}
  which we call SMP \textup(Sample Minmax Predictor\textup) whenever it exists\textup, in which case~\eqref{eq:main-excess-risk-stat} becomes
  \begin{equation}
    \label{eq:excess-risk-pred-stat}
    \E \big[ \excessrisk_\phi (\wt f_{\phi, n}) \big]
    \leq \E_{Z_1^n, X} \Big[ \inf_{\pred \in \predspace} \sup_{y \in \Y} 
    \Big\{ \ell (\pred, y) - \ell_\phi (\wh f_{\phi, n}^{(X,y)} (X), y) \Big\} 
    \Big] \, .
  \end{equation}
\end{theorem}

The proof of Theorem~\ref{thm:excess-risk-stat} is given in Section~\ref{sub:proofs-excess-general}.
The excess risk bound of Theorem~\ref{thm:excess-risk-stat} is related to the stability of the (regularized) empirical risk minimizer.
Indeed, if the ERM $\wh f_{\phi, n}^{(X,y)}$ obtained by adding a new sample $(X,y)$ does not depend too much on the label $y$, \ie if the set $\{ \wh f_{\phi, n}^{(X,y)} : y \in \Y \}$ is small in expectation, then the min-max quantity in the bound~\eqref{eq:excess-risk-pred-stat} will also be small.

Stability techniques were exploited by~\cite{bousquet2002stability} to establish guarantees for learning algorithms such as ERM or approximate ERM, although earlier instances of this method can be found in~\cite{vapnik1974theory,devroye1979holdout,haussler1994predicting}.
Stability arguments were used by~\cite{bousquet2002stability,shalev2010learnability} to prove fast rates of order $O(1/n)$ for ERM in strongly convex stochastic optimization problems and more recently by~\cite{koren2015expconcave} for exp-concave problems.
While related to the notion of stability, the excess risk bound of Theorem~\ref{thm:excess-risk-stat} differs from the latter stability bounds.
Indeed, approaches based on stability control the risk in terms of variations of the loss of the output hypothesis (such as ERM) under changes of the sample~\cite{bousquet2002stability,shalev2010learnability,srebro2010smoothness,koren2015expconcave}.
By contrast, Theorem~\ref{thm:excess-risk-stat} controls the risk in terms of some min-max quantity, which measures the size of the set of empirical risk minimizers obtained by adding one sample.
The difference between the two is most apparent in the context of logistic regression (see Section~\ref{sec:logistic} below), where it is critical to obtain improved guarantees that could not be derived from loss stability of regularized risk minimizers.

It is worth noting that the SMP~\eqref{eq:pred-rule-stat} whose risk is controlled in~\eqref{eq:excess-risk-pred-stat} is \emph{not} the regularized ERM, that is, the algorithm whose ``stability'' is controlled.
In fact, $\wt f_{\phi, n}$ is in general an \emph{improper} predictor, which does not belong to the class $\F$; it may be seen as a ``center'' of the set of risk minimizers obtained by adding one sample, in a sense related to the loss function.
In fact, we will show in what follows that SMP enjoys guarantees which are not achievable by proper predictors such as regularized ERM.

\subsection{Conditional density estimation with the logarithmic loss}
\label{sec:excess-risk-log}

We now turn to conditional density estimation, which is the focus of this work, by considering the logarithmic loss.
Let $\mu$ be a measure on $\Y$ and $\predspace$ be the set of probability densities on $\Y$ with respect to $\mu$, namely the set of measurable functions $g: \Y \to \R^+$ such that $\int_\Y g \di \mu = 1$.
The logarithmic loss is defined as $\ell (g, y) = - \log g(y)$ for $g \in \predspace$ and $y \in \Y$.
In this setting, a predictor $f: \X \to \predspace$ corresponds to a conditional density.
We denote $f (y \cond x) = f(x) (y)$ and as before $\ell (f, z) = \ell (f(x), y)$ for $z= (x,y)$.
In this case, the ERM~\eqref{eq:def-erm} corresponds to the (conditional) maximum likelihood estimator (MLE).
The risk of a conditional density $f$ is
\begin{equation*}
  R (f) = - \E \big[\log f (Y \cond X) \big]
\end{equation*}
whenever this expectation is defined.
For any conditional densities $f,g$ with respect to $\mu$, 
\begin{equation}
  \label{eq:diff-risk-log}
  R(g) - R(f)
  = \E \bigg[ \log \frac{f(Y \cond X)}{g (Y \cond X)} \bigg]
  \, ,
\end{equation}
which does not depend on the dominating measure $\mu$ (since the ratio $f/g$ does not), but only on the conditional distributions $f\mu,g \mu$ and the distribution of $X$.
In particular, we may choose $\mu$ such that the risk $R(f)$
is well-defined and finite for some $f \in \F$, and identify $f$ and $g$ with the corresponding conditional distributions.
There exists a best predictor $f^* \in \F$ whenever the excess risk $\excessrisk (f) = \E [ \ell (f, Z) - \ell (f^*, Z) ]$ is defined and belongs to $[0, +\infty]$ for every $f \in \F$.
Following what we did in Section~\ref{sec:excess-risk-statistical-learning}, given a penalization function $\phi : \F \to \R$, we define the penalized risk $R_\phi$ and the penalized excess risk $\excessrisk_\phi$.

Theorem~\ref{thm:excess-risk-log} below shows that both SMP defined in Theorem~\ref{thm:excess-risk-stat} and its excess risk bound~\eqref{eq:excess-risk-pred-stat} can be described explicitly in this case.

\begin{theorem}[Excess risk bound for conditional density estimation]
  \label{thm:excess-risk-log}
  In the case of the logarithmic loss, the SMP $\wt f_{\phi, n}$ defined in~\eqref{eq:pred-rule-stat} writes
  \begin{equation}
    \label{eq:estimator-log}
    \wt f_{\phi, n} (y \cond x)    
    = \frac{\wh f_{\phi, n}^{(x,y)} (y \cond x) e^{-\phi(\wh f_{\phi, n}^{(x,y)})}}{\int_{\Y} \wh f_{\phi, n}^{(x,y')} (y' \cond x) e^{-\phi(\wh f_{\phi, n}^{(x, y')})} \mu (\di y')},
  \end{equation}
  whenever the integral $\int_{\Y} \wh f_{\phi, n}^{(X,y)} (y \cond X) e^{-\phi(\wh f_{\phi, n}^{(X, y)})} \mu (\di y)$ is finite almost surely \textup(over $Z_1^n, X$\textup).
  In addition\textup, its excess risk bound~\eqref{eq:excess-risk-pred-stat} writes
  \begin{equation}
    \label{eq:excess-risk-log}
    \E \big[ \excessrisk_\phi (\wt f_{\phi, n}) \big]
    \leq \E_{Z_1^n,X} \Big[ \log \Big( \int_{\Y} \wh f_{\phi, n}^{(X,y)} (y \cond X) 
    e^{-\phi(\wh f_{\phi, n}^{(X, y)})} \mu (\di y) \Big) \Big]  \, .
  \end{equation}
\end{theorem}

\begin{remark}
  \label{rem:unpen-smp}
  In the non-regularized case where $\phi \equiv 0$, SMP simply writes
  \begin{equation*}
    \wt f_n (y \cond x)
    = \frac{\wh f_{n}^{(x,y)} (y \cond x)}{\int_{\Y} \wh f_{n}^{(x,y')} (y' \cond x) \mu (\di y')}
    \, ,
  \end{equation*}
  while its excess risk bound~\eqref{eq:excess-risk-log} takes the form:
  \begin{equation*}
    \E \big[ \excessrisk (\wt f_{n}) \big]
    \leq \E_{Z_1^n,X} \Big[ \log \Big( \int_{\Y} \wh f_{n}^{(X,y)} (y \cond X) \mu (\di y) \Big) \Big]
    \, .
  \end{equation*}
\end{remark}

Theorem~\ref{thm:excess-risk-log} is proved in Section~\ref{sub:proofs-excess-general}.
The SMP~\eqref{eq:estimator-log} minimizes, for every value of~$x$, 
the worst-case (over $y \in \Y$) excess loss $\ell (\wt f_{\phi, n}(x), y) - \ell_\phi (\wh f_{\phi, n}^{(x,y)} (x), y)$ with respect to the ERM on the sample $Z_1^n,(x,y)$.
As explained above, the right-hand side of~\eqref{eq:excess-risk-log} corresponds to (the expectation of) a measure of complexity of the class $\{ \wh f_{\phi, n}^{(x,y)}, y \in \Y \}$ associated to the $\log$-loss.
We will see below, in particular cases for $\F$, that despite being derived from a general bound for statistical learning, the excess risk bound of the SMP is quite tight and close to the optimal risk in the well-specified case.
In fact, we will see in the case of the Gaussian linear model (Section~\ref{sec:unreg-smp}) that the bound of the SMP is intrinsic to the hardness of the problem.

In the unconditional case, the prediction of the estimator~\eqref{eq:estimator-log} closely resembles that of a sequential prediction strategy called Sequential Normalized Maximum Likelihood (SNML), introduced by~\cite{roos2008snml} and related to the Last Step Minimax algorithm (which restricts to proper predictions) from~\cite{takimoto2000last}\footnote{Specifically, the prediction of SMP coincides with that of the SNML-1 algorithm from~\cite{roos2008snml} at step $n+1$, while SNML-2 from~\cite{roos2008snml} (simply called SNML in subsequent work \cite{kotlowski2011snml,bartlett2013snml}) is slightly different: it minimizes worst-case regret with respect to next step ERM on the whole sequence, instead of just the last~sample.}.
Interestingly, the motivation is completely different: the SNML algorithm was introduced as a computationally efficient relaxation of the minimax algorithm (in terms of cumulative regret) for sequential prediction under log-loss; its worst-case regret was shown to be almost minimax~\cite{kotlowski2011snml}, and in fact minimax for some specific families~\cite{bartlett2013snml}.
By contrast, in our case the SMP estimator naturally arises as the minimizer of a novel upper bound on the \emph{non-cumulative} excess risk.


Finally, the excess risk bound~\eqref{eq:excess-risk-log}, which is specific to the statistical learning setting, provides in this case a more ``local'' control than bounds derived from the sequential setting, such as the averaged Bayesian mixture approach~\cite{yang1999information,catoni2004statistical,juditsky2008mirror,audibert2009fastrates}, as well as the normalized maximum likelihood (NML) \cite{shtarkov1987universal,merhav1998universal,cesabianchi2006plg,grunwald2007mdl}.
This will be illustrated in the following sections.

We also note that a similar min-max approach was adopted by \cite{forster2002relative} to design a procedure for linear regression with square loss, see also \cite{vovk2001competitive,azoury2001relative} for a related method in the online setting.

\section{Some consequences for density estimation}
\label{sec:density-estimation}

In this section, we consider the problem of (unconditional) density estimation: the space $\X$ is assumed to be trivial (with a single element) and is thus omitted\footnote{While conditional density estimation can be cast as a special case of density estimation, we adopt the opposite perspective since SMP exploits the conditional structure.}, and no penalization is used ($\phi \equiv 0$).
In other words, given access to an \iid sample $(Y_1, \dots, Y_n)$ from a distribution $P$ on $\Y$, and given a family $\F$ of probability densities on $\Y$ with respect to $\mu$ (namely, a statistical model $\F$), the aim is to find a predictive distribution $\wh g_n$ on $\F$ whose excess risk with respect to $\F$ is as small as possible.
Note that the model may be \emph{misspecified}, in the sense that $P \not\in \F$.
Introduce the Kullback-Leibler (KL) divergence
\begin{equation*}
  \kll{P}{Q} = \E_{Y \sim P} \Big[ \log \frac{\di P}{\di Q} (Y) \Big]
\end{equation*}
between distributions $P$ and $Q$ (which is infinite whenever $P$ is not absolutely continuous with respect to $Q$). 
Let $f^* \in \argmin_{f \in \F} R (f)$ (assuming it exists); if $\kll{P}{f^*} < + \infty$ then $f^* = \argmin_{f \in \F} \kll{P}{f}$ and the excess risk~\eqref{eq:def-excess-risk-statlearn} writes $\excessrisk(f) = \kll{P}{f} - \kll{P}{f^*}$ for any $f \in \F$.
For this reason, the risk $R$ is also called \emph{KL risk}.

In the next sections, we apply Theorem~\ref{thm:excess-risk-log} to misspecified density estimation on standard families.
In each case, the SMP is explicit and the excess risk bound scales as $d/n$ irrespective of the true distribution~$P$.
These bounds are tight, since they are \emph{within a factor of} $2$ of the optimal asymptotic rate in the well-specified case.
Also, we compare it with MLE and online to batch conversion~\cite{cesabianchi2004online_to_batch} of sequential prediction strategies.
In all considered examples, SMP improves these estimators.

\subsection{Finite alphabet: the multinomial model}
\label{sec:finite-alphabet}

In this section, we assume that $\Y$ is a finite set with $d$ elements, $\mu$ is the counting measure and $\F = \{ (p(y))_{y \in \Y} \in \R_+^\Y : \sum_{y \in \Y} p(y) = 1 \}$ is the multinomial model (which is always well-specified).
For any $y \in \Y$, we let $N_n(y) = \sum_{i=1}^n \indic{Y_i = y}$.

\begin{proposition}
  \label{prop:multinomial}
  If $\Y$ is a finite set with $d$ elements\textup, then SMP corresponds to the \emph{Laplace estimator}
  \begin{equation}
    \label{eq:def-laplace}
    \wt f_n (y) = \frac{N_n(y) + 1}{n + d}\, .
  \end{equation}
  In addition\textup, the bound~\eqref{eq:excess-risk-log} writes in this case
  \begin{equation}
    \label{eq:excess-laplace}
    \E \big[ \excessrisk(\wt f_n) \big]
    \leq \log \Big( \frac{n + d}{n + 1} \Big)
    \leq \frac{d - 1}{n}
    \, .
  \end{equation}
\end{proposition}

Proposition~\ref{prop:multinomial} is proved in Section~\ref{sub:proofs-density-estimation}.
In this case, the SMP corresponds to the Laplace estimator, which is the Bayes predictive distribution under an uniform prior on $\F$.
The first bound in~\eqref{eq:excess-laplace} is tight: it is an equality when $Y$ is constant almost surely.

\paragraph{On MLE.}

The MLE is given by $\wh f_n (y) = N_n(y)/n$.
Its expected risk is infinite unless $P$ is concentrated on a single point.
Indeed, let $y_0, y_1 \in \Y$ be distinct elements such that $\P( Y = y_0), \P (Y=y_1) > 0$;
with positive probability, $Y_1 = \cdots = Y_n = y_0$, so that $\wh f_n (y) = \indic{y = y_0}$, $\ell (\wh f_n, y_1) = + \infty$ and thus $R (\wh f_n) = + \infty$.
Hence, $\E [R (\wh f_n) ] = + \infty$.
In order to obtain non-vacuous expected risk for MLE in this case, one may restrict to $\F_\delta = \{ p \in \F : \forall y \in \Y, \ p (y) \geq \delta \}$ for some $\delta \in (0, 1)$, so that log ratios of densities are bounded.
In this case, whenever $p \in \F_\delta$, the excess risk of MLE has asymptotically efficient rate $(d-1)/(2n) + o(n^{-1})$.
This reflects the fact that the model is well-specified.

\paragraph{On online to batch conversion.}

The minimax cumulative regret with respect to the class $\F$ scales (for fixed $d$ and as $n \to \infty$) as $(d-1) (\log n) /2 + O (1)$ \cite[Remark~9.2 p.~256]{cesabianchi2006plg}.
Hence, any upper bound based on online-to-batch conversion~\cite{cesabianchi2004online_to_batch} can be no better than $(d-1) (\log n) / (2n) + O(1/n)$.

\subsection{The Gaussian location model}
\label{sub:gaussian-family}

We now let $\Y = \R^d$ and consider the
Gaussian location model, namely the family $\F = \{ \gaussdist(\theta, \Sigma) : \theta \in \R^d \}$ of Gaussian distributions with fixed positive covariance matrix $\Sigma$.
We let $\bar Y_n := \frac 1n \sum_{i=1}^n Y_i$.

\begin{proposition}
  \label{prop:gaussian-location}
  A risk minimizer $f^* = \gaussdist(\theta^*, \Sigma) \in \F$ exists if and only if the random variable $Y \sim P$ satisfies $\E \| Y \| < + \infty$\textup, in which case $\theta^* = \E [Y]$.
  For $n \geq 1$\textup, the SMP is given by $\wt f_n = \gaussdist ( \bar{Y}_n, (1+1/n)^2 \Sigma)$\textup, and whenever $\E \| Y \| < + \infty$ the bound~\eqref{eq:excess-risk-log} writes
  \begin{equation}
    \label{eq:excessrisk-gaussian}
    \E \big[ \excessrisk (\wt f_n) \big]
    \leq d \log \Big( 1 + \frac{1}{n} \Big)
    \leq \frac{d}{n}
    \, .
  \end{equation}
  In addition\textup, when the model is well-specified\textup, we have
  \begin{equation*}
    \E [\excessrisk(\wt f_n)] = d \log \Big( 1 + \frac{1}{n} \Big) - \frac{d}{2n} < \frac{d}{2n}
    \, .
  \end{equation*}
\end{proposition}

The proof of Proposition~\ref{prop:gaussian-location} is given in Section~\ref{sub:proofs-density-estimation} below.
This bound is valid under misspecification: it does not depend on the distribution of $Y$, and is at most twice the optimal asymptotic risk $d/(2n)$.

\paragraph{On MLE and proper estimators.} 

Assume that $\E \| Y \|^2 < +\infty$ and define $\Sigma_Y = \E [ (Y - \E Y) (Y - \E Y)^{\top} ]$.
The excess risk of the MLE $\wh f_n = \gaussdist(\bar{Y}_n, \Sigma)$ is given by
\begin{equation*}
  \excessrisk(\wh f_n) = \frac{1}{2} \E  \big\| \bar{Y}_n - \E [Y]
   \big \|^2_{\Sigma^{-1}}   = \frac{1}{2n} \tr (\Sigma^{-1} 
   \Sigma_Y).
\end{equation*}
In the misspecified case where $\Sigma_Y \neq \Sigma$, this quantity depends on the true distribution of $Y$ and can be arbitrarily large depending on $\Sigma_Y$.
This limitation is in fact shared by any proper estimator of the form $f_{\wh \theta_n} = \gaussdist(\wh \theta_n, \Sigma)$ for some $\wh \theta_n$, as explained next.
Consider the family of distributions $\{ P_{\theta^*} = \gaussdist(\theta^*, \Sigma_Y) : \theta^* \in \R^d \}$ for some arbitrary symmetric positive matrix $\Sigma_Y$, and the loss function $L(\theta^*, \theta) = \| \theta - \theta^* \|^2_{\Sigma^{-1}} / 2$.
It is a standard result in decision theory (see \eg \cite{lehmann1998tpe}) that the empirical mean $\bar{Y}_n$ is minimax optimal for this problem and has constant risk $\tr (\Sigma^{-1} \Sigma_Y) / (2n)$.
Therefore, for any proper estimator $f_{\wh \theta_n}$,
\begin{equation*}
  \sup_{\theta^* \in \R^d} \E_{Y_1, \dots,Y_n \sim P_{\theta^*}^{\otimes n}} \big[ \excessrisk (f_{\wh \theta_n}) 
  \big]
  = \frac{1}{2} \sup_{\theta^* \in \R^d} \E_{\theta^*} \big\| \wh \theta_n - \E [ Y ] 
  \big\|^2_{\Sigma^{-1}}  \geq \frac{\tr (\Sigma^{-1} \Sigma_Y)}{2 n} \, .
\end{equation*}

\paragraph{On online to batch conversion.}

The minimax cumulative regret with respect to the full Gaussian family $\F$ is infinite (see, \eg, \cite[Example~11.1, p.~298]{grunwald2007mdl}):
this comes from the fact that regret after the first step (the first prediction being made before seeing any sample) is unbounded.
This difficulty does not appear in the batch setting, where one can predict conditionally on the sample, in a translation-invariant fashion.
One can guarantee finite minimax regret by considering a restricted model $\{ \gaussdist(\theta, \Sigma) : \theta \in K \}$ for some compact set $K \subset \R^d$~\cite{rissanen1996fisher}, in which case minimax regret scales as $d (\log n)/2 + C_K + o(1)$ (for a constant $C_K$ depending on $K$) so that online to batch conversion yields an excess risk bound of $d (\log n)/(2n) + C_K/n + o(1/n)$, which exhibits an extra $\log n$ factor and a dependence on $K$.

\paragraph{Exact minimax rate in the misspecified case.} 

In fact, for the Gaussian location family, the minimax excess risk in the general misspecified case, namely
\begin{equation}
  \label{eq:minimax-excessrisk}
  \inf_{\wh g_n} \sup_{P} \E_{Y_1, \dots, Y_n \sim P^{\otimes n}} \big[ 
  \excessrisk (\wh g_n) \big]
\end{equation}
where the supremum is over all probability distributions $P$ on $\R^d$ with $\E \| Y \|^2 < + \infty$, the infimum over density estimators $\wh g_n$ and where the excess risk is under the true distribution $P$, can be determined exactly, together with a minimax estimator, as shown below.

\begin{theorem}
  \label{thm:gauss-minimax}
  For the Gaussian location model, the minimax excess risk~\eqref{eq:minimax-excessrisk} in the misspecified case \textup(namely, over all distributions with finite second moment\textup) is equal to
  \begin{equation*}
    \inf_{\wh g_n} \sup_{P} \E_{Y_1, \dots, Y_n \sim P^{\otimes n}} \big[ 
    \excessrisk (\wh g_n) \big] = \frac{d}{2} \log 
    \Big(1 + \frac{1}{n} \Big)
    \, .
  \end{equation*}
  In addition, the minimax excess risk is achieved by the estimator $\wh g_n = \gaussdist(\bar{Y}_n, (1+1/n) \Sigma)$, which satisfies $\E [ \excessrisk(\wh g_n) ] = ({d}/{2}) \log ( 1 + {1}/{n})$ for any distribution $P$ of $Y$ such that $\E [ \| Y \|^2 ] < + \infty$.
\end{theorem}

Theorem~\ref{thm:gauss-minimax} is proven in Section~\ref{sub:proofs-density-estimation} below.
Note that $\wh g_n$ corresponds to the Bayes predictive posterior under uniform prior, which is known to achieve the minimax risk in the \emph{well-specified} case~\cite{ng1980estimation,murray1977density}, see also~\cite{george2006improved}.
Remarkably, both the minimax excess risk and the minimax estimator remain the same in the misspecified case.
This holds even though the posterior itself (a distribution on $\F$) does not concentrate on a neighborhood of the best parameter $\theta^* = \E [Y]$ in the misspecified case (contrary to the well-specified case), when the true variance is large.
An explanation for this phenomenon is that the out-of-model correction of the Bayes predictive posterior (critically due to averaging over the posterior) brings it closer to distributions with high variance, thereby compensating the high variability for such distributions.
As a result, the Bayes predictive posterior equalizes the excess risk across all distributions.
This suggests that posterior concentration rates alone, which do not take
into account the latter effect (and degrade under model misspecification when the true variance is large), fail to accurately characterize the excess risk of predictive posteriors under model misspecification.

Finally, Theorem~\ref{thm:gauss-minimax} shows that the worst-case excess risk bound~\eqref{eq:excessrisk-gaussian} of SMP is exactly twice the minimax excess risk for distributions with finite variance.

\section{Gaussian linear conditional density estimation}
\label{sec:cond-gauss}

In this section, we turn to conditional density estimation, starting with arguably the most standard family, namely the linear Gaussian model.
After introducing the setting, notation and basic assumptions (Section~\ref{sec:setting-linear}), we consider the non-penalized SMP and its excess risk bounds with respect to the full unrestricted model (Section~\ref{sec:unreg-smp}).
Next, we consider in Section~\ref{sec:ridge-variant} the Ridge-regularized SMP and its performance, both in the finite-dimensional context and in the nonparametric one where $d$ may be larger than $n$.
In the latter case, the bounds only depend on the covariance structure of $X$ and on the norm of the comparison parameter.

\subsection{Setting: the Gaussian linear model}
\label{sec:setting-linear}

Consider the spaces $\X = \R^d$ and $\Y = \R$ and the family of conditional distributions
\begin{equation}
  \label{eq:def-cond-gauss-family}
  \F = \big \{ f_{\theta} (\cdot \cond x) = \gaussdist (\langle \theta, x\rangle, \sigma^2) : \theta \in \R^d \big\}
\end{equation}
for some $\sigma^2 > 0$; up to the change of variables $y' = y/\sigma$, we will assume without loss of generality that $\sigma^2 = 1$.
Throughout this section, we consider log-loss with respect to the base measure $\mu = (2\pi)^{-1/2} \di y$ on $\R$, so that for $\theta \in \R^d$ and $(x, y) \in \R^d \times \R$:
\begin{equation}
  \label{eq:logloss-linear-gaussian}
  \ell (f_{\theta}, (x,y))
  = - \log f_{\theta} (y \cond x)
  = \frac{1}{2} (y - \langle \theta, x\rangle)^2
  \, ,
\end{equation}
and hence the risk of $f_\theta$ writes
\begin{equation*}
  R (f_\theta) = \frac{1}{2} \E \big[ (Y - \langle \theta, X\rangle)^2 \big] \, .
\end{equation*}
The problem of conditional density estimation in the Gaussian linear model is intimately linked (but not equivalent) to that of linear least-squares regression, namely statistical learning with the square loss and a comparison class formed by linear predictors.
Let us discuss the connection and differences between the two problems:
\begin{itemize}
\item In the least-squares problem, one is interested in a \emph{point prediction} of the response $y$ given the covariates $x$, or equivalently in an estimate of the \emph{conditional expectation} $\E [Y \cond X]$ of $Y$ given $X$.
  By contrast, in density estimation one seeks a \emph{probabilistic prediction} of $y$ given $x$, or equivalently an estimate of the \emph{conditional distribution} of $Y$ given $X$, which includes a quantification of the uncertainty of $Y$ given $X$.
\item When one restricts to proper, within-model estimators (taking values in $\F$),
  the two problems are equivalent, as shown by the expression of the loss~\eqref{eq:logloss-linear-gaussian}.
\item On the other hand, in the context of conditional density estimation, the possibility of using improper (out-of-model) estimators provides more flexibility.
  As we will see, this additional flexibility is essential to bypass lower bounds for proper estimators in the misspecified case.
\end{itemize}

Let us emphasize that in the context of conditional density estimation, well-specification refers to the fact that the conditional distribution of $Y$ given $X$ belongs to the model.
As in the unconditional case, we are interested in bounds that do not degrade under model misspecification, and hence require only weak assumptions on this conditional distribution.
Assumption~\ref{ass:finite-variance-X-and-Y} below will be made throughout this section, while further assumptions will be made in Sections~\ref{sec:unreg-smp} and~\ref{sec:ridge-variant} respectively.

\begin{assumption}[Finite second moments]
  \label{ass:finite-variance-X-and-Y}
  Both $X$ and $Y$ are square integrable, namely
  \begin{equation*}
    \E \| X \|^2 < + \infty \quad \text{ and } \quad \E [Y^2] < + \infty.
  \end{equation*}
\end{assumption}

We denote $\Sigma = \Sigma_X = \E [X X^\top]$ the second moment matrix, which we call (following a common abuse of terminology) the \emph{covariance matrix} of $X$, even when $X$ is not centered.
Assumption~\ref{ass:finite-variance-X-and-Y} implies that $Y X$ is integrable (by the Cauchy-Schwarz inequality) and that $\E [ \langle \theta, X\rangle^2] = \langle \Sigma \theta, \theta \rangle$, so that the risk $R (f_\theta)$ is finite\footnote{The assumption $\E [Y^2] < + \infty$ is not strictly necessary to ensure that $R (f_\theta)$ is finite for some base measure~$\mu$.
  Indeed, taking $\mu = \gaussdist (0, 1)$, log-loss writes $\ell (f_{\theta}, (x,y)) = \langle \theta, x\rangle^2/2 - y \langle \theta, x\rangle$, and the slightly weaker assumption that $Y X$ is integrable suffices.
  We nonetheless take a uniform dominating measure $\mu$ and make Assumption~\ref{ass:finite-variance-X-and-Y}, in order to make the connection with the least-squares problem more explicit.} 
and equals:
\begin{equation*}
  R (f_\theta)
  = \frac{1}{2} \langle \Sigma \theta, \theta \rangle - \langle \theta, \E [Y X] \rangle + \frac{1}{2} \E [Y^2]
  \, ,
\end{equation*}
with gradient $\nabla R (f_\theta) = \Sigma \theta - \E [Y X]$.
In particular, whenever $\Sigma$ is invertible, the population risk minimizer $f^* \in \F$ is given by $f^* = f_{\theta^*}$ with $\theta^* = \Sigma^{-1} \E [Y X]$, while the excess risk of $f_\theta \in \F$ writes $\excessrisk (f_\theta) = \frac{1}{2} \left\| \theta - \theta^* \right\|_\Sigma^2$.
Likewise, whenever the empirical covariance matrix
\begin{equation}
  \label{eq:def-empirical-covariance}
  \wh \Sigma_n := \frac 1n \sum_{i=1}^n X_i X_i^\top
\end{equation}
is invertible, there exists a unique empirical risk minimizer given by
\begin{equation}
  \label{eq:def-mle-ols}
  \wh \theta_n
  = \argmin_{\theta \in \R^d} \sum_{i=1}^n (Y_i - \langle \theta, X_i\rangle)^2
  = \wh \Sigma_n^{-1} \wh S_n
\end{equation}
where $\wh S_n = n^{-1} \sum_{i=1}^n Y_i X_i$.
Hence, whenever $\wh \Sigma_n$ is invertible (almost surely), the MLE is uniquely defined, and equals the \emph{ordinary least squares} estimator given by~\eqref{eq:def-mle-ols}.

\subsection{The unregularized SMP}
\label{sec:unreg-smp}

In this section, we consider uniform excess risk bounds for unpenalized SMP ($\phi \equiv 0$) with respect to the linear Gaussian class $\F$ given by~\eqref{eq:def-cond-gauss-family}.
This setting is relevant when $n \gg d$, especially when little is known or assumed on the optimal parameter $\theta^*$.
We will work under the following

\begin{assumption}[Non-degenerate design]
  \label{ass:weak-X}
  The covariance matrix $\Sigma$ is invertible and the 
  empirical covariance matrix $\wh \Sigma_n$ is invertible almost surely.
\end{assumption}

The fact that $\Sigma$ is invertible amounts to assuming that $X$ is not supported in any hyperplane of $\R^d$.
This assumption is not restrictive, since otherwise one can simply restrict to the span of the support of $X$, a subspace of $\R^d$; we make it merely for convenience in statements and notation.
In addition, a simple induction~\cite{mourtada2019leastsquares} shows that Assumption~\ref{ass:weak-X} amounts to assuming that $n \geq d$ and that $\P (X \in H) = 0$ for any hyperplane $H \subset \R^d$.
Note that the latter is granted whenever $X$ admits a density with respect to the Lebesgue measure.
Moreover, as explained in Section~\ref{sec:setting-linear}, Assumption~\ref{ass:weak-X} amounts to saying that MLE in the model~\eqref{eq:def-cond-gauss-family} is uniquely determined almost surely.

Once again in this case, SMP leads to an improper estimator, which can be made explicit and satisfies a sharp excess risk bound.
Let us introduce the rescaled empirical covariance matrix
\begin{equation}
  \label{eq:def-rescaled_covariance}
  \wt \Sigma_n = \Sigma^{-1/2} \wh \Sigma_n \Sigma^{-1/2} = \frac 1n  \sum_{i=1}^n \wt X_i \wt X_i^\top \quad \text{ where } \quad \wt X_i = \Sigma^{-1/2} X_i.
\end{equation}
Note that the rescaled design $\wt X_i$ is such that
$\E [\wt X_i \wt X_i^\top ] = \id$ for $i = 1, \dots, n$.
As explained in Theorem~\ref{thm:smp-gaussian-linear} below, the excess risk of SMP is connected to the fluctuations of $\wt \Sigma_n$.

\begin{theorem}
  \label{thm:smp-gaussian-linear}
  Assume that Assumptions~\ref{ass:finite-variance-X-and-Y} and~\ref{ass:weak-X} are fulfilled.
  For the Gaussian linear family $\F$ given by~\eqref{eq:def-cond-gauss-family}\textup, SMP is given by
  \begin{equation}
    \label{eq:smp-linear-gaussian}
    \wt f_n (\cdot \cond x) = \gaussdist \Big( \langle \wh \theta_n, x \rangle, \big(1 + \big \langle (n \wh \Sigma_n)^{-1} x, x \big \rangle \big)^2 \Big).
  \end{equation}
  In addition\textup, it satisfies the following excess risk bound\textup:
  \begin{equation}
    \label{eq:excessrisk-smp-linear}
    \E \big[ \excessrisk (\wt f_n) \big] \leq 
    \E \Big[ - \log \Big( 1 - \big \langle (n \wh\Sigma_n + X X^\top)^{-1} X, X \big \rangle \Big) \Big]
    \leq \log \Big( 1 + \frac{1}{n} \E \big[ \tr ( \wt \Sigma_n^{-1}) \big] \Big) \, ,
  \end{equation}
  where $\wt \Sigma_n$ is the rescaled empirical covariance given by~\eqref{eq:def-rescaled_covariance}.
\end{theorem}

The proof of Theorem~\ref{thm:smp-gaussian-linear} is given in Section~\ref{sub:proof-cond-gauss-line} below.
The upper bound on the excess risk depends on the distribution of the design through the term $\E[ \tr ( \wt \Sigma_n^{-1} )]$, namely through lower relative fluctuations of the empirical covariance matrix $\wh \Sigma_n$ with respect to its population counterpart $\Sigma$.
Note that this quantity is invariant under linear transformation of $X, X_1, \ldots, X_n$.

A key feature of the excess risk bound~\eqref{eq:excessrisk-smp-linear} on the SMP is that it only depends
on the distribution of $X$, and \emph{not} on the conditional distribution of $Y$ given $X$.
The expected risk of the SMP is therefore not affected by model misspecification, similarly to what was observed in Section~\ref{sec:density-estimation} for unconditional densities.
This is once again a strong departure from the behavior of the MLE, as explained below.

\paragraph{Comparison with MLE and proper estimators.}

As explained above, MLE is given by $f_{\wh \theta_n}$, where $\wh \theta_n$ is the ordinary least-squares estimator~\eqref{eq:def-mle-ols}.
In the \emph{well-specified case}, the minimax risk among \emph{proper} estimators is achieved by MLE and equals $\E [ \tr (\wt \Sigma_n^{-1})] / (2n)$ \cite{mourtada2019leastsquares};
hence, the excess risk of SMP is only {within a factor $2$} of the  minimax risk for proper estimators in the well-specified case, despite the fact that the model can be misspecified.
In the \emph{misspecified case}, the risk of MLE scales as 
$\E_{(X, Y) \sim P} [ (Y - \langle \theta^*, X \rangle)^2 \| \Sigma^{-1/2} X \|^2] / n$ up to lower-order terms, and this dependence is unavoidable for any proper estimator \cite{mourtada2019leastsquares}.
This means that the risk of proper estimators deteriorates under misspecification, and that the minimax risk among proper estimators is infinite, since the previous quantity can be arbitrarily large.

\paragraph{Comparison with the well-specified case.}

One can in fact show that the first bound in~\eqref{eq:excessrisk-smp-linear} on the risk of SMP in the general misspecified case is exactly \emph{twice} the minimax excess risk in the well-specified case, for any distribution $P_X$ of covariates.
This shows that the general excess risk bound for SMP is intrinsic to the complexity of the problem in this case.
Another consequence worth pointing is that the minimax excess risk in the misspecified case is at most twice that of the well-specified case.

\paragraph{Comparison with online algorithms.}

The minimax regret with respect to the full linear model is infinite, since regret after the first observation is unbounded.
Hence, one cannot obtain any uniform excess risk bound from online-to-batch conversion of sequential procedures.
We discuss non-uniform guarantees in Section~\ref{sec:ridge-variant}.

\paragraph{Link with leverage scores.}

It is worth noting that the first part of the upper bound~\eqref{eq:excessrisk-smp-linear} has a natural interpretation.
Indeed, the quantity $\langle (n \wh\Sigma_n + X X^\top )^{-1} X, X \rangle$ is the \emph{leverage score} of $X$ in the sample $X_1, \dots, X_n, X$.
This means that the excess risk of SMP can be upper bounded as 
\begin{equation*}
  \E \big[ \excessrisk (\wt f_n) \big] \leq \E \big[ - \log (1 - \wh \ell_{n+1}) \big], \quad \text{where} \quad \wh \ell_{n+1} = \bigg\langle \bigg( \sum_{i=1}^{n+1} X_i X_i^\top \bigg)^{-1} X_{n+1}, X_{n+1} \bigg\rangle 
\end{equation*}
is the leverage score of one sample distributed as $P_X$ among $n+1$.
Intuitively, the more uneven the leverage scores are, the harder the prediction task will be, since the optimal parameter in the model will effectively be determined by smaller number of points and hence have larger variance.

\paragraph{Upper bounds.}

A first upper bound on the risk of the SMP can be obtained from~\eqref{eq:excessrisk-smp-linear} in the case of {Gaussian covariates:}
when $X \sim \gaussdist (0, \Sigma)$, so that $\wt X \sim \gaussdist (0, \id)$, we have $\E [ \tr (\wt \Sigma_n^{-1})] = n d / (n - d - 1)$ \cite{anderson2003introduction,breiman1983many}, giving an upper bound of $\log (1 + d / (n-d-1))$ for SMP.

We now discuss extensions to general distributions $P_X$ of covariates.
By the law of large numbers, one has $\wt \Sigma_n \to \id$ as $n \to \infty$ and thus $\tr (\wt \Sigma_n^{-1}) \to d$ almost surely.
Hence, one can expect that the excess risk bound~\eqref{eq:excessrisk-smp-linear} of the SMP scales as $d/n + o(1/n)$.
In order to turn this into an explicit, non-asymptotic bound, we need to control the lower tail of $\wt \Sigma_n$.
This requires some conditions on the distribution of $X$, in order to ensure even finiteness of $\E [ \tr (\wt \Sigma_n^{-1}) ]$: 

\begin{assumption}[Small ball]
  \label{ass:small-ball}
  There exist constants $C \geq 1$ and $\alpha \in (0, 1)$ such that, for any hyperplane $H \subset \R^d$ and $t > 0$,
  \begin{equation}
    \label{eq:small-ball-hyp}
    \P ( \dist (\Sigma^{-1/2} X, H) \leq t)
    \leq (C t)^\alpha
    \, .
  \end{equation}  
\end{assumption}
Assumption~\ref{ass:small-ball} quantifies Assumption~\ref{ass:weak-X}, which states that $\P(X \in H) = 0$ for any hyperplane $H \subset \R^d$.
It is equivalent to $\P ( | \langle \theta, X\rangle | \leq t \| \theta \|_{\Sigma} ) \leq (C t)^\alpha$ for every $\theta \in \R^d$ and $t \in (0, 1)$.
This condition is a strengthened version of the \emph{small-ball condition} considered in~\cite{koltchinskii2015smallest,mendelson2015smallball,lecue2016performance},
which amounts to requiring this for a single $t < C^{-1}$.
A matching lower bound to~\eqref{eq:small-ball-hyp} holds with $\alpha = 1$ and $C = 0.025$ for any distribution of $X$ when $d \geq 2$ \cite{mourtada2019leastsquares}.

\begin{assumption}[Kurtosis]
  \label{ass:kurtosis}
  $\E \| \Sigma^{-1/2} X \|^4 \leq \kappa d^2$ for some $\kappa \geq 1$.
\end{assumption}

Assumption~\ref{ass:kurtosis} is a bound on the kurtosis of $\| \Sigma^{-1/2} X \|$, since $\E \| \Sigma^{-1/2} X \|^2 = d$.
It is weaker than the following $L^2$--$L^4$ equivalence for one-dimensional marginals of~$X$: $(\E \langle X, \theta \rangle^4)^{1/4} \leq \kappa^{1/4} (\E \langle X, \theta \rangle^2)^{1/2}$ for all $\theta \in \R^d$ \cite{oliveira2016covariance}, and a significantly weaker requirement on $X$ than a sub-Gaussian assumption~\cite{vershynin2012introduction}.

\begin{corollary}
  \label{cor:smp-linear-nonasymptotic}
  Suppose that Assumptions~\ref{ass:finite-variance-X-and-Y}\textup,~\ref{ass:weak-X}\textup,~\ref{ass:small-ball} and~\ref{ass:kurtosis} hold, and let $\wt f_n$ be the SMP given by~\eqref{eq:smp-linear-gaussian}.
  Then\textup, denoting $C' = 28 C^4 e^{1 + 9 / \alpha}$, for $n \geq \min(6 d / \alpha, 12 \log(12 / \alpha) / \alpha)$ we have 
  \begin{equation}
    \label{eq:smp-linear-nonasymptotic}
    \E \big[ \excessrisk (\wt f_n) \big] \leq \frac{d}{n} 
    \Big(1 + C' \frac{\kappa d}{n}  \Big)
    \, .
  \end{equation}
\end{corollary}

The proof of Corollary~\ref{cor:smp-linear-nonasymptotic} is given in Section~\ref{sec:proofs}. 
It is a direct consequence of Theorem~\ref{thm:smp-gaussian-linear}, together with an upper bound from~\cite{mourtada2019leastsquares} on the excess risk of the ordinary least-squares estimator in the well-specified case.
The bound~\eqref{eq:smp-linear-nonasymptotic} deduced from Theorem~\ref{thm:smp-gaussian-linear} scales as $d / n + O((d / n)^2)$ as $d = o (n)$, with exact first-order constant and order-optimal second-order term $O((d / n)^{2})$.
The most technical argument is provided in~\cite{mourtada2019leastsquares},
where a tight control on the smallest eigenvalue of $\wt \Sigma_n$ and on $\E [\tr (\wt \Sigma_n^{-1}) ]$ is obtained under Assumptions~\ref{ass:small-ball} and~\ref{ass:kurtosis}.

\subsection{Ridge-regularized SMP}
\label{sec:ridge-variant}

In the previous section, we considered uniform excess risk bounds with respect to the full Gaussian linear model $\F$.
We now turn to non-uniform bounds over $\F$, where some dependence on the comparison parameter $\theta \in \R^d$ is allowed.
Such guarantees are relevant when uniform bounds over $\F$ are not possible, which occurs either when $d > n$, or when the distribution of covariates $X$ does not satisfy the regularity condition (Assumption~\ref{ass:weak-X} or~\ref{ass:small-ball}) ensuring finite minimax risk.

Specifically, we investigate excess risk bounds with respect to balls of the form $\F_B = \{ f_\theta : \| \theta \| \leq B \}$ for some $B > 0$.
For this purpose, we will consider SMP with Ridge regularization $\phi (\theta) = \lambda \| \theta \|^2 / 2$ for some $\lambda > 0$.
One advantage of the bounds obtained in this setting is that
they remain meaningful in the \emph{nonparametric} setting where $d$ may be larger than~$n$.

The upper bound from Theorem~\ref{thm:smp-linear-ridge} below does not explicitly depend on the dimension $d$, but only on the covariance matrix $\Sigma$ and on $\| \theta \|$.
It extends readily to the case where $\R^d$ is replaced by a Reproducing Kernel Hilbert Space (RKHS) $\H$, but we keep $\R^d$ in order to keep the setting and notation consistent with those of Section~\ref{sec:unreg-smp}.
We work in this section under the following assumption.

\begin{assumption}[Bounded covariates]
  \label{ass:bounded-covariates}
  $\| X\| \leq R$ almost surely for some constant $R > 0$.
\end{assumption}

Assumption~\ref{ass:bounded-covariates} is automatically satisfied for instance in the Reproducing Kernel Hilbert Space (RKHS) setting, where the features $x$ are of the form $x = \Phi (x')$ where $x' \in \X'$ is an input variable in some measurable space $\X'$ and $\Phi: \X' \to \R^d$ a measurable map such that the kernel $K : \X' \times \X' \to \R$ given by $K (x', x'') = \langle \Phi (x'), \Phi (x'') \rangle$ is bounded: $K \leq R^2$.

Recall that we consider the family $\F = \{ f_{\theta}(\cdot \cond x) = \gaussdist (\langle \theta, x \rangle, 1) : \theta \in \R^d \}$, together with the Ridge penalization $\phi (\theta) = {\lambda} \| \theta \|^2 / 2$ for some $\lambda > 0$.
Let 
\begin{equation*}
  \wh \theta_{\lambda, n}
  := \argmin_{\theta \in \R^d} \bigg\{ \frac{1}{n} \sum_{i=1}^n \ell (f_\theta, (X_i, Y_i)) + \frac{\lambda}{2} \| \theta \|^2 \bigg\}
  = (\wh \Sigma_n + \lambda \id)^{-1} \wh S_n
\end{equation*}
denote the Ridge estimator, where we recall that  $\wh \Sigma_n = n^{-1} \sum_{i=1}^n X_i X_i^\top$ and $\wh S_n = n^{-1} \sum_{i=1}^n Y_i X_i$, and let us also define 
\begin{equation*}
  \wh \Sigma_\lambda^x = n \wh \Sigma_n + x x^\top + \lambda (n+1) \id, \quad \wh K_\lambda^x = ( \wh \Sigma_\lambda^x )^{-1} \quad \text{and} \quad \lambda' = \frac{n+1}{n} \lambda
  \, .
\end{equation*}
We also introduce the \emph{degrees of freedom} of the Ridge estimator \cite{wahba1990spline,friedman2001elements,wasserman2006nonparametric}, given by
\begin{equation}
  \label{eq:def-df}
  \dflambda{\Sigma} = \tr [ (\Sigma + \lambda \id)^{-1} \Sigma]
  \, ,
\end{equation}
and note that 
\begin{equation}
  \label{eq:dflambda-smaller-d}
  \dflambda{\Sigma} \leq \tr [ (\Sigma + \lambda \id)^{-1} (\Sigma + \lambda \id)] = d
  \, .
\end{equation}

\begin{theorem}
  \label{thm:smp-linear-ridge}
  Let $\lambda > 0$.
  The penalized SMP~\eqref{eq:estimator-log} with penalty $\phi(\theta) = \frac{\lambda}{2} \| \theta \|^2$ is well-defined and writes $\wt f_{\lambda, n} (\cdot \cond x) = \gaussdist (\wt \mu_\lambda (x), \wt \sigma_\lambda^2 (x))$\textup, where
  \begin{equation}
    \label{eq:smp-linear-ridge-sigma}
    \wt \sigma_\lambda(x)^2
    = \big( {( 1 - \| x \|_{\wh K_\lambda^x}^2 )^2 + \lambda \| x \|_{(\wh K_\lambda^x)^2}^2 } \big)^{-1}
  \end{equation}
  and
  \begin{equation}
    \label{eq:smp-linear-ridge-mu}
    \wt \mu_\lambda(x) = \langle \wh \theta_{\lambda', n}, x \rangle - \lambda 
    \wt \sigma_\lambda(x)^2  \langle \wh \theta_{\lambda', n}, x \rangle_{\wh K_\lambda^x}.
  \end{equation}
  In addition\textup, under Assumptions~\ref{ass:finite-variance-X-and-Y} and~\ref{ass:bounded-covariates}\textup, we have
  \begin{equation}
    \label{eq:excessrisk-smp-linear-ridge}
    \E \big[ R (\wt f_{\lambda, n}) \big]
    - \inf_{\theta \in \R^d} \Big\{ R (f_\theta) + \frac{\lambda}{2} \| \theta \|^2  \Big\}
    \leq 1.25 \cdot \frac{\dflambda{\Sigma}}{n+1}
  \end{equation}
  for every $\lambda \geq 2 R^2 / (n+1)$\textup, where $\dflambda{\Sigma}$ is given by~\eqref{eq:def-df}.
\end{theorem}

Although the space of parameters is finite-dimensional (of dimension $d$), the bound~\eqref{eq:excessrisk-smp-linear-ridge} is ``non-parametric'' in the sense that it does not feature any explicit dependence on $d$; rather, it only depends on the spectral properties of $\Sigma$ through $\dflambda{\Sigma}$.
In particular, it remains nonvacuous even when $d \gg n$; in fact, as mentioned above, Theorem~\ref{thm:smp-linear-ridge} remains valid (with the same proof, up to minor changes in terminology and notation) in the case of an infinite-dimensional RKHS.
Its proof, which is provided in Section~\ref{sub:proof-cond-gauss-line}, relies on a combination of the general SMP excess risk bound, exchangeability, and convexity of matrix functions.

Let us now discuss some consequences of Theorem~\ref{thm:smp-linear-ridge}.

\begin{itemize}
\item \emph{Finite-dimensional case.}
  Since $\dflambda{\Sigma} \leq d$ (see~\eqref{eq:dflambda-smaller-d}), Theorem~\ref{thm:smp-linear-ridge} entails, for $\lambda = 2 R^2 / (n+1)$, that 
\begin{equation}
  \label{eq:excessrisk-ridge-finitedim}
  \E [ R (\wt f_{\lambda, n}) ] - \inf_{\| \theta \| \leq B} R (f_\theta)
  \leq \frac{1.25  d + B^2 R^2}{n+1}
\end{equation}
for every $B > 0$.
This gives an excess risk bound of $O ( (d + B^2 R^2)/n )$.
Proposition~\ref{prop:smp-ridge-log-norm} below further refines this finite-dimensional bound.
\item \emph{Slow, dimension-free rate.} 
Since $\dflambda{\Sigma} \leq \tr (\Sigma) / \lambda \leq R^2 / \lambda$ for $\lambda > 0$, Theorem~\ref{thm:smp-linear-ridge} yields, for every $\lambda \geq 2 R^2 / (n+1)$ and $B > 0$,
\begin{equation}
  \label{eq:excessrisk-slow}
  \E [ R (\wt f_{\lambda, n}) ] - \inf_{\| \theta \| \leq B} R (f_\theta)  
  \leq \frac{1.25 R^2}{\lambda (n + 1)} + \frac{\lambda B^2}{2}
  \leq \frac{2 B R}{\sqrt{n}} + \frac{B^2 R^2}{n},
\end{equation}
where the second inequality is obtained with $\lambda = \max (2 R^2/(n+1), 2R /(B \sqrt{n+1}) )$.
This corresponds to the standard nonparametric slow rate for regression, except that it does not depend on the range of $Y$.
This requires no assumption on the covariance $\Sigma$, aside from the inequality $\tr (\Sigma) \leq R^2$ implied by the assumption $\| X \| \leq R$.
\item \emph{Nonparametric case.}
More precise results can be obtained in terms of spectral properties of $\Sigma$.
Let $b > 1$ be a rate of decay of the eigenvalues of $\Sigma$, such that $\dflambda{\Sigma} \leq C \lambda^{-1/b}$.
Then, Theorem~\ref{thm:smp-linear-ridge} yields
\begin{equation}
  \label{eq:excessrisk-capacity}
  \E \big[ R(\wt f_{\lambda, n}) \big] - \inf_{\| \theta \| \leq B} R (f_\theta)
  \lesssim \frac{C \lambda^{-1/b}}{n} + \lambda B^2 
  \asymp C^{b/(b+1)} B^{2/(b+1)} n^{-b/(b+1)}
\end{equation}
for $\lambda \asymp (B^2 n/C)^{- b / (b+1)}$.
This matches the {minimax} rate for regression with unit noise over balls of RKHSs in the well-specified case, without additional assumptions on $\theta$ \cite{caponnetto2007optimal}.
\end{itemize}

In the finite-dimensional case where $n \gg d$, one can improve the quadratic dependence on the norm $B = \| \theta \|$.
This yields bounds that are appropriate when the covariate distribution is possibly degenerate, in the sense that Assumption~\ref{ass:weak-X} does not hold, so that excess risk bounds uniform in $\theta$ are no longer achievable.

\begin{proposition}
  \label{prop:smp-ridge-log-norm}
  Grant Assumptions~\ref{ass:finite-variance-X-and-Y} and~\ref{ass:bounded-covariates}.
  Then\textup, for any $B > 0$\textup, the Ridge-SMP $\wt f_{\lambda, n}$ of Theorem~\ref{thm:smp-linear-ridge}
  with $\lambda = d / ( B^2 (n+1))$ satisfies
  \begin{equation}
    \label{eq:smp-ridge-log-norm}
    \E \big[ R (\wt f_{\lambda, n}) \big] - \inf_{\theta \in \R^d \pp \| \theta \| \leq B} R (f_\theta)
    \leq \frac{5 d \log \big( 2 + {B R}/{\sqrt{d}} \big)}{n + 1}
    \, .
  \end{equation}
\end{proposition}

This bound is of order $O (d \log (BR/\sqrt{d}) / n)$.
It improves a guarantee obtained by~\cite{kakade2005online} (with suitable parameters and online-to-batch conversion) of $O (d \log (B^2 R^2 n / d) / n)$ from the sequential setting through Bayesian mixture strategies, by removing an extra $O (\log n)$ term (the bound from \cite{kakade2005online} holds in the online setting, where the $\log n$ term is unavoidable).
In addition, the bound from Theorem~\ref{thm:smp-linear-ridge} is dimension-free, whereas the guarantee from \cite{kakade2005online} has explicit dependence on $d$.

\begin{remark}[Parameter scaling]
  \label{rem:param-scaling}
  The previous results are valid for arbitrary parameters $BR,d,n$.
  In order to make these bounds more concrete, we now discuss some natural scaling for the norm $BR$.
  Consider the finite-dimensional case where $n \gg d$, and assume that $\Sigma$ is well-conditioned, in the sense that $c := \opnorm{\Sigma} \cdot \opnorm{\Sigma^{-1}} = O (1)$.
  This means that $X$ is approximately isotropic, or equivalently that the chosen norm on $\R^d$ does not favor specific directions, but rather controls signal strength $\| \theta \|_\Sigma \asymp \| \theta \|$; this can be ensured in practice by rescaling covariates.
  Also, assume that $\| \Sigma^{-1/2} X \| \leq \rho \sqrt{d}$ for some $\rho \geq 1$, a bounded leverage condition~\cite{hsu2014ridge}, and let $\psi := \| \theta \|_\Sigma = \E [ \langle \theta, X\rangle^2 ]^{1/2}$ denote signal strength.
  Then,
  \begin{align*}
    \| \theta \| \cdot \| X \|
    \leq \opnorm{\Sigma^{-1/2}} \cdot \| \Sigma^{1/2} \theta \| \cdot  \opnorm{\Sigma^{1/2}} \cdot \| \Sigma^{-1/2} X \| 
    \leq {c}^{1/2} \rho \psi \sqrt{d}
    \, ,
  \end{align*}
  so that $B R \leq {c}^{1/2} \rho \psi \sqrt{d} = O (\sqrt{d})$.
  
  On the other hand, one can have $BR \ll \sqrt{d}$: this occurs in the ``nonparametric'' case where $\Sigma$ has eigenvalue decay, and $\theta$ lies close to the space spanned by the leading eigenvectors of $\Sigma$; in this case, $\dflambda{\Sigma} \ll d$, and it is beneficial to replace $d$ by $\dflambda{\Sigma}$ as in Theorem~\ref{thm:smp-linear-ridge}.
\end{remark}

We close this section by pointing out that, in the well-conditioned finite-dimensional regime where $B R = O (\sqrt{d})$, the bounds~\eqref{eq:excessrisk-ridge-finitedim} and~\eqref{eq:smp-ridge-log-norm} both yield a $O (d/n)$ guarantee, while the latter has an improved dependence on signal strength.

\section{Logistic regression}
\label{sec:logistic}

In this section, we consider conditional density estimation with a binary response, using the logistic model.
Section~\ref{sec:setting-logistic} introduces the setting.
We consider the unpenalized SMP ($\phi \equiv 0$) in Section~\ref{sec:smp-logist-regr} and contrast its predictions with those of MLE.
In Section~\ref{sec:excess-risk-logistic} we introduce the Logistic SMP procedure with Ridge penalization, and establish a non-asymptotic bound on its excess risk.

\subsection{Setting}
\label{sec:setting-logistic}

We consider binary labels in $\Y = \{ -1, 1 \}$, with counting measure $\mu = \delta_0 + \delta_1$, while $\X = \R^d$.
The \emph{logistic model} is the family of conditional distributions given by
\begin{equation}
  \label{eq:def-logistic-density}
  \F = \{ f_{\theta} : \theta \in \R^d \}, \quad \text{where} \quad 
  f_{\theta} (1 \cond x) := 1 - f_{\theta} (-1 \cond x) = \sigma ( \langle \theta, x \rangle)
\end{equation}
for any $x \in \R^d$, with $\sigma(u) = e^u / (1 + e^{u})$ for $u \in \R$ the \emph{sigmoid} function.
Since $\sigma (-u) = 1 - \sigma (u)$, one simply has $f_{\theta} (y \cond x) = \sigma ( y \langle \theta, x \rangle)$ for $x \in \R^d$ and $y \in \{ -1, 1\}$.
The log-loss of $f_\theta \in \F$ at a sample $(x, y) \in \R^d \times \{ -1, 1 \}$ writes
\begin{equation}
  \label{eq:loss-logistic-beta}
  \ell (f_\theta, (x,y))
  = - \log f_\theta (y \cond x)
  = \log( 1 + e^{-y \langle \theta, x\rangle})
  = \ell (- y \langle \theta, x\rangle)
  \, ,
\end{equation}
where we introduced the \emph{logistic} loss $\ell(u) = \log (1 + e^u)$ for $u \in \R$.
Let $(X, Y)$ have distribution $P$ on $\R^d \times \{ -1, 1 \}$, such that $\E \| X \| < + \infty$.
Since $\ell'(u) = \sigma (u) \in [0, 1]$ for any $u \in \R$, we have
$0 \leq \ell (u) \leq \log 2 + |u|$ so that
$\ell (-Y\langle \theta, X\rangle) 
\leq \log 2 + \| \theta \| \|X \|$, and the risk of $f_\theta$, namely
\begin{equation}
  \label{eq:risk-logistic}
  R (f_\theta)
  = \E [ \ell (- Y\langle \theta, X\rangle) ] \, ,
\end{equation}
is well-defined.
Given a sample $(X_i, Y_i)$, $1 \leq i \leq n$, a MLE $\wh \theta_n$ is given by
\begin{equation}
  \label{eq:mle-logistic}
  \wh \theta_n \in \argmin_{\theta \in \R^d} \frac{1}{n} \sum_{i=1}^n \ell (- Y_i\langle \theta, X_i\rangle)\, ,
\end{equation}
A MLE~\eqref{eq:mle-logistic} does not always exist, and may not be unique.
Indeed, it is a well-known fact (see~\cite{candes2020phase} for recent results on this topic in the high-dimensional regime) that there is no MLE~\eqref{eq:mle-logistic} whenever the sets $\{ X_i : Y_i = 1 \}$ and $\{ X_i : Y_i = - 1 \}$ are strictly \emph{linearly separated} by a hyperplane, namely when one can find $\theta \in \R^d$ such that $Y_i \langle \theta, X_i\rangle > 0$ for all $i=1, \ldots, n$
(indeed, in this case the empirical risk of $t \theta$ converges to $0$ as $t \to + \infty$, while the empirical risk is positive on $\R^d$).
In addition, when a MLE exists in $\R^d$, one can see that it is unique if and only if $V = \Span(X_1, \dots, X_n) = \R^d$: in this case, the empirical risk is strictly convex on $\R^d$ since $\ell: \R \to \R$ is.

It is convenient to enrich the class $\F$ given by~\eqref{eq:def-logistic-density} to ensure existence (though not uniqueness) of MLE in the separated case.
Specifically, define the model $\overline{\F}$ obtained by adding to $\F$ the conditional densities ${f}_{\infty,\theta}$ for $\theta \in \R^d$, $\| \theta \| = 1$, defined by $f_{\infty,\theta} (1 \cond x) = 1$ if $\langle \theta, x\rangle > 0$, $0$ if $\langle \theta, x\rangle < 0$ and $1/2$ if $\langle \theta, x\rangle = 0$.
Denote by $\overline{\Theta}$ the parameter space obtained by adding to $\R^d$ elements of the form $(\infty, \theta)$.
We note that MLE exists in $\overline{\F}$ in the separated case, although it is not unique since it depends on the choice of a separating hyperplane defined by $\theta$.
Given a choice of MLE, we let
\begin{equation}
  \label{eq:mle-sample-logistic}
    \wh \theta_n^{(x,y)}
    = \argmin_{\theta \in \overline{\Theta}
    } \Big\{ \sum_{i=1}^n \ell (f_\theta, (X_i, Y_i)) + \ell (f_\theta, (x, y)) \Big\} 
\end{equation}
for any $(x, y) \in \R^d \times \{ -1, 1\}$.
It is also convenient to let $Z_i = - Y_i X_i$; then, one has $\wh \theta_n^{(x,y)} = \wh \theta_n^{-yx}$, where for $z \in \R^d$ we define (with a slight abuse of notation for $\theta \in \overline{\Theta} \setminus \R^d$)
\begin{equation}
  \label{eq:mle-sample-logistic-Z}
  \wh \theta_n^z
  = \argmin_{\theta \in \overline{\Theta}} \Big\{ \sum_{i=1}^n \ell (\langle \theta, Z_i\rangle) + \ell (\langle \theta, z\rangle) \Big\}
  \, .
\end{equation}

\subsection{SMP for logistic regression}
\label{sec:smp-logist-regr}

Let us now instantiate SMP as well as Theorem~\ref{thm:excess-risk-log} to the logistic family.

\begin{proposition}
  \label{prop:logistic}
  For the family of logistic conditional distributions~\eqref{eq:def-logistic-density}\textup, SMP writes
  \begin{equation}
    \label{eq:predictor-logistic}
    \wt f_n (y \cond x)
    = \frac{f_{\wh \theta_n^{(x,y)}} (y \cond x)}{f_{\wh \theta_n^{(x,1)}} (1 \cond x) + f_{\wh \theta_n^{(x,-1)}} (-1 \cond x)}
    = \frac{\sigma (\langle \wh \theta_n^{(x,y)}, yx \rangle)}{\sigma (\langle \wh \theta_n^{(x,1)}, x \rangle) + \sigma (\langle \wh \theta_n^{(x,-1)}, - x \rangle)}
  \end{equation}
  for every $x \in \R^d$ and $y \in \{ -1, 1\}$. 
  Unlike the MLE~\eqref{eq:mle-sample-logistic}\textup, SMP is always well-defined and unique.
  We always have that $\wt f_n (y \cond x) \in (0, 1)$ and it does not depend on the choice of a MLE in the linearly separated case.
  In addition\textup, it satisfies the following excess risk bound\textup:
  \begin{equation}
    \label{eq:excessrisk-logistic-exact}
    \E \big[ \excessrisk (\wt f_n) \big]
    \leq \E_{Z_1^n, Z} \big[ \sigma (\langle \wh \theta_n^{-Z}, Z \rangle) - \sigma (\langle \wh \theta_n^{Z}, Z\rangle) \big] \, ,
  \end{equation}
  where $Z_1, \dots, Z_n, Z$ are \iid variables distributed as $- Y X$.
\end{proposition}

The proof of Proposition~\ref{prop:logistic} is given in Section~\ref{sec:proofs-logistic} below.
Unlike MLE, SMP is always well-defined and outputs predictions in $(0, 1)$.
Indeed, the numerator in~\eqref{eq:predictor-logistic} belongs to $(0, 1]$, and whenever the points $Y_1 X_1, \dots, Y_n X_n, y x$ belong to a half-space passing through the origin (so that MLE does not exist in $\R^d$), we have
$f_{\wh \theta_n^{(x,y)}} (y \cond x) = 1$, so that the prediction of SMP is well-defined and does not depend on the choice of MLE in~\eqref{eq:mle-sample-logistic}, see the proof of Proposition~\ref{prop:logistic} for details.

\paragraph{Comparison with MLE.}

SMP corrects a well-known deficiency of MLE, which tends to produce overly confident and ill-calibrated predictions~\cite{sur2019modern}.
To emphasize this effect, consider the case of a point $x$ for which the virtual datasets $(X_1, Y_1), \dots, (X_n, Y_n), (x,y)$ are separated for both $y = -1$ and $y = 1$.
Then, the prediction $\wh f_n (1 \cond x)$ of an MLE $\wh f_n \in \overline{\F}$ can be either $1$ or $0$, both being possible depending on the specific choice of separating hyperplane.
Hence, in this case the prediction of MLE is both highly confident and dependent on an arbitrary choice.
By contrast, in this situation SMP gives equal probability $1/2$ to both classes, reflecting the uncertainty for such points $x$.

\paragraph{A non-Bayesian approach to calibration.}

As for the Gaussian linear model (Section~\ref{sec:cond-gauss}), SMP returns more uncertain conditional distributions for input points $x$ with high ``leverage'', namely strong influence on the prediction of MLE at this point.
This provides a simple and natural approach to calibration of probabilistic predictions for logistic regression, which does not rely on Bayesian methods.
Such an approach is appealing on computational grounds, since the prediction $\wt f (\cdot \cond x)$ of SMP is obtained by solving two logistic regressions~\eqref{eq:mle-sample-logistic} and does not require approximate posterior sampling.

\paragraph{Comparison with stability approaches.}

Approaches based on stability of the loss \cite{bousquet2002stability,shalev2010learnability,srebro2010smoothness,koren2015expconcave} would lead to a control of the excess risk involving $\ell (\langle \wh \theta_n^{-Z}, Z\rangle) - \ell (\langle \wh \theta_n^{Z}, Z\rangle)$, while Proposition~\ref{prop:logistic} involves 
$\sigma (\langle \wh \theta_n^{-Z}, Z \rangle) - \sigma (\langle \wh \theta_n^{Z}, Z \rangle)$, where we recall that $\ell (u) = \log (1 + e^u)$ and $\sigma(u) = 1 / (1 + e^{-u})$.
Whenever $u' \approx u \gg 1$, we have $\ell (u') - \ell (u) \approx \ell' (u) \cdot (u' - u)
\approx u' - u$, while $\sigma (u') - \sigma (u) \approx \sigma' (u) \cdot (u' - u) \approx e^{-u} \cdot (u' - u)$.
In this case, the SMP bound is exponentially smaller than the loss stability bound.
This roughly explains why we are able to remove terms of order $e^{BR}$ from our upper bound on the excess risk of SMP, provided in the next section.

\subsection{Excess risk bounds for Ridge-regularized SMP}
\label{sec:excess-risk-logistic}

In order to obtain explicit and precise non-asymptotic guarantees, we consider a Ridge-regularized variant of SMP for logistic regression.
Specifically, for $\lambda > 0$ we consider the penalty $\phi (\theta) = \lambda \| \theta \|^2 / 2$.
The corresponding penalized SMP can be computed as follows: for every $z \in \R^d$, let
\begin{equation}
  \label{eq:max-sample-ridge-logistic}
  \wh \theta_{\lambda, n}^{z}
  := \argmin_{\theta \in \R^d} \bigg\{ \frac{1}{n+1} \Big( \sum_{i=1}^n \ell (\langle \theta, Z_i\rangle) + \ell (\langle \theta, z\rangle) 
  \Big) + \frac{\lambda}{2} \| \theta \|^2 \bigg\} \, .
\end{equation}
Note that $\wh \theta_{\lambda, n}^{z} \in \R^d$ exists and is unique, since the regularized objective in~\eqref{eq:max-sample-ridge-logistic} is strongly convex, hence strictly convex and diverging as $\| \theta \| \to +\infty$.
As before, we let $\wh \theta_{\lambda, n}^{(x,y)} = \wh \theta_{\lambda, n}^{-yx}$ for $(x, y) \in \R^d \times \{ -1, 1\}$.
Now, following Theorem~\ref{thm:excess-risk-log}, the regularized SMP writes in this case
\begin{equation}
  \label{eq:smp-ridge-logistic}
  \wt f_{\lambda, n} (y \cond x)
  = \frac{\sigma (y \langle \wh \theta_{\lambda, n}^{(x,y)}, x\rangle) \, e^{-\lambda \| \wh \theta_{\lambda, n}^{(x,y)} \|^2/2 }}{\sigma (\langle \wh \theta_{\lambda, n}^{(x,1)}, x\rangle) \, e^{-\lambda \| \wh \theta_{\lambda, n}^{(x,1)} \|^2/2 } + \sigma ( -\langle \wh \theta_{\lambda, n}^{(x,-1)}, x\rangle) \, e^{-\lambda \| \wh \theta_{\lambda, n}^{(x,-1)} \|^2/2 }}
\end{equation}
for any $(x, y) \in \R^d \times \{ -1, 1\}$, and comes as before at the cost of two ridge-regularized logistic regressions.

We will work under Assumption~\ref{ass:bounded-covariates}, namely $\|X\| \leq R$ almost surely, as in Section~\ref{sec:ridge-variant} for the Gaussian linear model.
Our main guarantee for Ridge-regularized SMP is stated in a nonparametric setting, where dependence on the dimension $d$ is kept implicit through the degrees of freedom~\eqref{eq:def-df}.

\begin{theorem}
  \label{thm:logistic-ridge-smp}
  Grant Assumption~\ref{ass:bounded-covariates}\textup, and
  assume that $\lambda \geq 2 R^2 / (n+1)$.
  Then\textup, the Ridge-regularized logistic SMP given by~\eqref{eq:smp-ridge-logistic} satisfies
  \begin{equation}
    \label{eq:excessrisk-logistic-ridge-smp}
    \E \big[ R (\wt f_{\lambda, n}) \big]
    - \inf_{\theta \in \R^d} \Big\{ R (f_\theta) + \frac{\lambda}{2} \| \theta \|^2 \Big\}
    \leq e \cdot \frac{\df{4\lambda}{\Sigma}}{n}
    \, ,
  \end{equation}
  where we recall that $\df{\lambda}{\Sigma} =
   \tr [ (\Sigma + \lambda I)^{-1} \Sigma ]$.
\end{theorem}

The upper bound~\eqref{eq:excessrisk-logistic-ridge-smp} is a \emph{fast rate} excess risk guarantee; it is worth noting that it only requires bounded covariates (Assumption~\ref{ass:bounded-covariates}).
In particular, it requires no assumption on the conditional distribution of $Y$ given $X$.
Furthermore, when the feature $X$ comes from a bounded kernel (see the discussion in Section~\ref{sec:ridge-variant} above), the bound~\eqref{eq:excessrisk-logistic-ridge-smp} is valid {under no assumption} on the distribution of $(X, Y)$.

We note that~\cite{marteau2019beyond} established nonparametric fast rate guarantees akin to~\eqref{eq:excessrisk-logistic-ridge-smp} for the Ridge-regularized estimator in the \emph{well-specified} case.
Compared to~\eqref{eq:excessrisk-logistic-ridge-smp}, their bias term, while also equal to $\lambda B^2$ under the sole assumption $\| \theta \| \leq B$, can be further improved under stronger assumptions on $\theta$
(namely, faster coefficient decay, or \emph{source condition}~\cite{caponnetto2007optimal}).
On the other hand, this result relies on the assumption of a well-specified model, and under our general assumptions such rates would exhibit exponential dependence in $BR$ \cite{hazan2014logistic}.

Since $\df{4 \lambda}{\Sigma} \leq d$ for every $\lambda$, we deduce the following result in finite dimension.

\begin{corollary}
  \label{cor:logistic-risk-smp-finitedim}
  Under Assumption~\ref{ass:bounded-covariates}\textup, the Ridge-regularized logistic SMP $\wt f_{\lambda, n}$~\eqref{eq:smp-ridge-logistic} with $\lambda = 2 R^2/ (n+1)$ satisfies, for every $B > 0$,
  \begin{equation}
    \label{eq:excessrisk-logistic-ridge-smp-finitedim}
    \E [ R (\wt f_{\lambda, n}) ]
    - \inf_{\| \theta \| \leq B} R (f_\theta)
    \leq \frac{e \cdot d + B^2 R^2}{n}
    \, .
  \end{equation}
\end{corollary}

Note that under the well-conditioned scaling of dimension $d$ with constant signal strength, namely $B R = O (\sqrt{d})$ (see Remark~\ref{rem:param-scaling} from Section~\ref{sec:ridge-variant}),
Corollary~\ref{cor:logistic-risk-smp-finitedim} yields an excess risk of $O (d / n)$.

\paragraph{Comparison with proper estimators.}

Under Assumption~\ref{ass:bounded-covariates}, Corollary~\ref{cor:logistic-risk-smp-finitedim} leads to an upper bound for Ridge SMP of  $O ( (d + B^2 R^2) / n )$ with respect to the ball $\| \theta \| \leq B$.
By contrast, \cite{hazan2014logistic} showed a lower bound for any \emph{proper} estimator (including the norm-constrained or Ridge-penalized MLE, or any stochastic optimization procedure
) of order $\min (BR/ \sqrt{n}, d e^{BR} / n)$ in the worst case.
We note that SMP is an improper estimator, as the log-odds ratio $\log ( \wt f_{\lambda, n}(1\cond x) / \wt f_{\lambda, n} (-1\cond x) )$ is nonlinear in $x$, and that it bypasses the lower bound for proper estimators.

\paragraph{A practical improper estimator.}

Fast rates of order $O (d \log (BR n) / n)$ are obtained by \cite{kakade2005online,foster2018logistic} under Assumption~\ref{ass:bounded-covariates}, by applying online-to-offline conversion (averaging) to a Bayes mixture sequential procedure, with prior on $\theta$ uniform over the ball of radius $B$ \cite{foster2018logistic} or Gaussian \cite{kakade2005online}.
This bound has an even better dependence on $B$
(logarithmic instead of quadratic) than Corollary~\ref{cor:logistic-risk-smp-finitedim}, although it also has a slightly worse dependence in $n$ (additional $\log n$ factor);
Theorem~\ref{thm:logistic-ridge-smp} additionally replaces $d$ by $\df{4\lambda}{\Sigma}$.
The main advantage of SMP over Bayes is that it is computationally less demanding: it replaces a problem of posterior sampling by one of optimization, since it requires training
two updated logistic regressions, 
starting for instance at the Ridge-penalized MLE.
Therefore, we partly answer an open problem from~\cite{foster2018logistic}, about finding an efficient alternative with fast rate, at least in the batch statistical learning case.
We note however that SMP is still more computationally demanding at prediction time than MLE, because of the required updates of the logistic risk minimization problem.

\paragraph{Overview of learning rates for logistic regression.}

Logistic regression with bounded features $\| X \| \leq R$ over the ball $\F_B = \{ f_\theta : \| \theta \| \leq B \}$ is (when restricting to proper estimators) a convex and $R$-Lipschitz stochastic optimization problem over a bounded domain.
This implies that a \emph{slow rate} of $O (BR/\sqrt{n})$ can be achieved by properly-tuned averaged projected online gradient descent~\cite{robbins1951stochastic,zinkevitch2003onlineconvex,shalevshwartz2012online,bubeck2015convex,hazan2016online}, Ridge-regularized ERM over $\R^d$ \cite{bousquet2002stability,sridharan2009fast}, or (as a linear prediction problem) constrained ERM over $\F_B$ \cite{kakade2009complexity,bartlett2002rademacher,meir2003generalization}.
Under the same assumptions, the logistic loss is $e^{-BR}$-exp-concave over $\F_B$, implying that a rate of $O (d e^{BR}/ n)$ can be achieved (up to potential $\log n$ factors) through the (averaged) Exponential Weights~\cite{cover1991portfolio,hazan2007logarithmic,vovk1998mixability} or Online Newton Step algorithms~\cite{hazan2007logarithmic,mahdavi2015expconcave}, as well as ERM over $\F_B$~\cite{koren2015expconcave,gonen2018stability,mehta2017expconcave_statistical}.
The improved dependence on $n$ in this bound is typically outweighed by the prohibitive exponential dependence on parameter norm.
As mentioned before, a lower bound of~\cite{hazan2014logistic} shows that, without further assumptions, no proper (within model $\F$) estimator can improve over the $O (\min (BR/\sqrt{n}, d e^{BR}/n))$ guarantee.
In order to bypass this lower bound, one has to resort to improper procedures~\cite{foster2018logistic}.
This is the approach taken by~\cite{foster2018logistic,kakade2005online} and ourselves, enabling improved guarantees without further assumptions, as discussed above.

Another line of work~\cite{bach2010logistic,bach2014logistic,bach2013nonstrongly,ostrovskii2018finite,marteau2019beyond} studies the behavior of specific (within-model) estimators, such as Ridge-regularized MLE or stochastic approximation procedures, in a distribution-dependent fashion.
A key technique in these refined analyses is the use of (generalized) \emph{self-concordance} of logistic loss, introduced by~\cite{bach2010logistic}, namely a control of the third derivative in terms of the second.
Following progress in~\cite{bach2010logistic,bach2014logistic}, \cite{bach2013nonstrongly} introduces a stochastic approximation algorithm with excess risk $O ( \rho^3 d (BR)^4 / n )$, where $\rho$ is a distribution-dependent curvature parameter.
This bound eliminates dependence on the smallest eigenvalue of the Hessian at the optimum~\cite{bach2014logistic}, but does not lead to the correct scaling in the finite-dimensional case with $BR = O (\sqrt{d})$, or in the nonparametric setting due to dependence on $d$ instead of $\dflambda{\Sigma}$ (see Remark~\ref{rem:param-scaling}).
In finite dimension, a tight non-asymptotic guarantee for MLE is obtained by~\cite{ostrovskii2018finite}, with an excess risk of $O (\deff / n)$ for $n \gtrsim \max (\rho \deff, d \log d)$, where $\deff$ denotes the \emph{effective dimension} characterizing the asymptotic risk of MLE~\eqref{eq:asymp-risk-mle-misspecified}.
These results are extended in~\cite{marteau2019beyond} in the well-specified nonparametric setting, with sharp risk bounds for the Ridge-regularized MLE.
In the worst case, the distribution-dependent constants $\rho$ and $\deff$ scale with $e^{BR}$~\cite{bach2013nonstrongly}, although they can be much smaller for more favorable distributions.
Despite the difference in assumptions, from a technical point of view, our analysis of the bound on the SMP excess risk also uses self-concordance.

In addition to these non-asymptotic analyses, a recent line of work~\cite{sur2019modern,barbier2019phase,candes2020phase,salehi2019impact} studies logistic regression under high-dimensional asymptotics where $d \asymp n$.
This asymptotic approach differs from the non-asymptotic one in that it provides an exact characterization of the error, but under highly specific distributional assumptions (well-specified model and Gaussian or jointly independent features).

\section{Conclusion}
\label{sec:conclusion}

In this paper, we derive a general excess risk bound for predictive density estimation under logarithmic loss.
Minimizing this bound naturally leads to a new improper (out-of-model) procedure, which we call \emph{Sample Minmax Predictor} (SMP).
On several problems, we show that the resulting bound, which is based on a refinement of the stability argument tailored for the logarithmic loss, scales as $d/n$, irrespective of the true distribution.
This contrasts with estimators taking values within the model, whose performance typically degrade under misspecification.
This estimator provides an alternative to approaches based on online-to-offline conversion \cite{barron1987bayes,catoni2004statistical,cesabianchi2004online_to_batch,audibert2009fastrates} of sequential procedures, whose rates feature an additional logarithmic dependence on sample size, and may be infinite for unbounded models.

We apply SMP to the Gaussian linear model.
In this case, it can be described explicitly, and achieves in the general misspecified case at most twice the minimax risk in the well-specified case, for every distribution of covariates.
We then consider a Ridge-regularized variant, which achieves nonparametric and finite-dimensional fast rates.

We then consider logistic regression.
Here, (Ridge-regularized) SMP is a simple and explicit procedure, whose predictions can be computed at the cost of two logistic regressions.
From a statistical perspective, it achieves fast excess risk rates even for worst-case distributions; such guarantees are known to be out of reach for any \emph{proper} procedure \cite{hazan2014logistic}.
In the statistical learning setting, this provides a practical alternative to the improper estimator from~\cite{foster2018logistic}, which relies on
Bayesian mixtures.
This work leaves a number of open questions and future directions:
\begin{itemize}
\item First, the excess risk bounds in this paper only hold in expectation, and not with exponential probability.
  This limitation is shared by procedures relying on online-to-batch conversion \cite{catoni2004statistical,audibert2008deviation,audibert2009fastrates,foster2018logistic}.
  In particular, the high-probability bound stated by \cite{foster2018logistic} for a procedure based on a ``confidence boosting'' technique \cite{mehta2017expconcave_statistical} appears to be incorrect:
  specifically, Equation~(17) herein is obtained by applying Markov's inequality to the excess risk; however, this quantity can take negative values since the predictor is outside the class.
  Designing efficient procedures that achieve high (exponential) probability excess risk bounds for misspecified logistic regression is an interesting direction for future work.
\item Second, it could be interesting
  to adapt the proposed method to online logistic regression, with a regret bound for individual sequences.
  Following this work, \cite{jezequel2020efficient} proposed a related (though distinct) practical sequential algorithm also relying on virtual samples, with a per-round regret of $O (d BR \log (BR n) / n)$.
  In finite dimension with $B R = O (\sqrt{d})$, this implies a $\wt O (d \sqrt{d}/ n )$ bound up to logarithmic terms, leaving room for further improvement in the online setting.
\item Another possibility is to apply SMP to other (conditional or otherwise) models beyond the Gaussian linear and logistic ones considered here, including other generalized linear models~\cite{mccullagh1989glm}.
\item Finally, it would be interesting to investigate conditions on the model and prior under which Bayes predictive posteriors (without iterate averaging) achieve uniform non-asymptotic bounds (such as Theorem~\ref{thm:gauss-minimax} or our guarantees for SMP).
\end{itemize}

On a more general note, density estimation under Kullback-Leibler risk possesses specific properties, which can be exploited to obtain more precise results than generic approaches applicable to general loss functions (which often suffer from the unboundedness of logarithmic loss).
This has been successfully exploited in the sequential case, where
cumulative criteria are considered.
Beyond the present work, we expect that further advances are possible in the statistical setting.
In particular, the idea of using a notion of leverage to quantify uncertainty for conditional density estimation, as done by SMP, as an alternative to Bayesian posteriors, may have broader applicability beyond the models we consider.

\section{Proofs} 
\label{sec:proofs}

\subsection{Proofs of general excess risk bounds (Section~\ref{sec:general-excess-risk})}
\label{sub:proofs-excess-general}

\begin{proof}[Proof of Theorem~\ref{thm:excess-risk-stat}]
  Let $Z_1^n,Z$ denote $n+1$ \iid variables distributed as $P$.
  We have
  \begin{align*}
    \E \big[ \excessrisk_\phi (\wh g_n) \big]
    &= \E_{Z_1^n, Z} [\ell (\wh g_n, Z)] - \inf_{f \in \F} \E_{Z_1^n, Z} \bigg[ \frac{1}{n+1} 
    \bigg\{ \sum_{i=1}^n \ell_\phi (f, Z_i) + \ell_\phi (f, Z) \bigg\} \bigg]  \\
    &= \E_{Z_1^n, Z} [\ell (\wh g_n, Z)] - \E_{Z_1^n, Z} \bigg[ \inf_{f \in \F} \frac{1}{n+1} 
    \bigg\{ \sum_{i=1}^n \ell_\phi (f, Z_i) + \ell_\phi (f, Z) \bigg\} \bigg] - \Delta_n
  \end{align*}
  where we denoted
  \begin{equation}
    \label{eq:proof-excessrisk-diff}
    \Delta_n = \inf_{f \in \F} \E \bigg[ \frac{1}{n+1} \bigg\{ \sum_{i=1}^n \ell_\phi (f, Z_i) + 
    \ell_\phi (f, Z) \bigg\} \bigg]
    - \E \bigg[ \inf_{f \in \F} \frac{1}{n+1} \bigg\{ \sum_{i=1}^n \ell_\phi (f, Z_i) + \ell_\phi (f, Z) \bigg\} \bigg] \geq 0  .
  \end{equation}
  In particular, by definition of $\wh f_{\phi, n}^Z$,
  \begin{equation}
    \label{eq:proof-excessrisk-2}
    \E \big[ \excessrisk_\phi (\wh g_n) \big]
    + \Delta_n
    = \E_{Z_1^n, Z} [\ell (\wh g_n, Z)] - \frac{1}{n+1} \E \bigg[ \sum_{i=1}^n \ell_\phi (\wh f_{\phi, n}^Z, Z_i) + \ell_\phi (\wh f_{\phi, n}^Z, Z) \bigg] \, .
  \end{equation}
  Since the distribution of the \iid sample $(Z_1, \dots, Z_n,Z)$ is preserved by exchanging $Z$ and $Z_i$, we have $\E [\ell_\phi (\wh f_{\phi, n}^Z, Z_i)] = \E [\ell_\phi (\wh f_{\phi, n}^Z, Z)]$ for $i=1, \ldots, n$ (recall that $\wh f_{\phi, n}^Z$ is chosen symmetrically in $Z_1, \dots, Z_n, Z$).
  Hence,~\eqref{eq:proof-excessrisk-2} becomes
  \begin{align}
    \label{eq:proof-excessrisk-3}
    \E \big[ \excessrisk_\phi (\wh g_n) \big] + \Delta_n    
    &= \E_{Z_1^n, Z} \big[ \ell (\wh g_n, Z) - \ell_\phi (\wh f_{\phi, n}^Z, Z) \big] \nonumber \\
    &= \E_{Z_1^n, X} \E_{Y | X} \big[ \ell (\wh g_n (X), Y) - 
    \ell_\phi (\wh f_{\phi, n}^{(X,Y)} (X), Y) \big] \nonumber \\
    &\leq \E_{Z_1^n, X} \Big[ \sup_{y \in \Y} \big \{ \ell (\wh g_n (X), y) 
    - \ell_\phi (\wh f_{\phi, n}^{(X,y)} (X), y) \big \} \Big]
      \, ,
  \end{align}
  which implies the bound~\eqref{eq:main-excess-risk-stat} since $\Delta_n \geq 0$.
  The remaining claims follow directly.
\end{proof}

\begin{proof}[Proof of Theorem~\ref{thm:excess-risk-log}]
  In the case of the logarithmic loss $\ell (p, (x,y)) = - \log p(y \cond x)$, we have for every density $p$ on $\Y$ and $x\in \X$:
  \begin{equation}
    \label{eq:proof-excessrisk-log-1}
    \sup_{y \in \Y} \big\{ \ell (p, y) - \ell_\phi (\wh f_{\phi, n}^{(x,y)}(x), y) \big\}
    = \sup_{y \in \Y} \log \frac{\wh f_{\phi, n}^{(x,y)} (y \cond x) e^{-\phi(\wh f_{\phi, n}^{(x,y)})}}{p (y)}
    \, .
  \end{equation}
  Now, Theorem~\ref{thm:excess-risk-log} follows from Theorem~\ref{thm:excess-risk-stat} together with Lemma~\ref{lem:minimax-logloss} below, where we consider  $g(y) = \wh f_{\phi, n}^{(x,y)} (y \cond x) e^{-\phi(\wh f_{\phi, n}^{(x,y)})}$.
\end{proof}

\begin{lemma}
  \label{lem:minimax-logloss}
  Let $g : \Y \to [0, +\infty]$ be a measurable function such that
  $\int_{\Y} g \di \mu \in \R_+^*$.
  Then,
  \begin{equation}
    \label{eq:minimax-logloss}
    \inf_p \sup_{y \in \Y} \log \frac{g(y)}{p(y)}
    = \log \bigg( \int_{\Y} g(y) \mu (\di y) \bigg)
    \, ,
  \end{equation}
  where the infimum in~\eqref{eq:minimax-logloss} spans over all probability densities $p: \Y \to \R^+$ with respect to $\mu$, and the infimum is reached at
  \begin{equation}
    \label{eq:minimax-logloss-argmin}
    p^* = \frac{g}{\int_{\Y} g \di \mu}
    \, .
  \end{equation}
\end{lemma}

\begin{proof}
  For every density $p$, denote $C(p) = \sup_{y \in \Y} \log g(y)/p(y)$.
  By definition, $p(y) \geq e^{-C(p)} g(y)$, so that since $p$ is a density
  \begin{equation*}
    1 = \int_{\Y} p(y) \mu (\di y)
    \geq e^{-C(p)} \int_\Y g (y) \mu(\di y)
    \, ,
  \end{equation*}
  so that $C(p) \geq \log \big( \int_\Y g \di \mu \big)$.
  Since $C(p^*) = \log \big( \int_\Y g \di \mu \big)$, this concludes the proof.
\end{proof}

We will sometimes also use the following observation:

\begin{lemma}
  \label{lem:exact-excessrisk-smp}
  The expected excess risk of the SMP is equal to\textup:
  \begin{equation}
    \label{eq:exact-excessrisk-smp}
    \E \big[ \excessrisk_\phi (\wt f_{\phi, n}) \big]
    = \E_{Z_1^n,X} \Big[ \log \Big( \int_{\Y} \wh f_{\phi, n}^{(X,y)} (y \cond X) e^{-\phi(\wh f_{\phi, n}^{(X,y)})} \mu (\di y) \Big) \Big] - \Delta_n
    \, ,
  \end{equation}
  where\textup, letting $Z_1, \dots, Z_{n+1}$ be \iid sample from $P$ and $f^*$ a risk minimizer \textup(when it exists\textup)\textup,
  \begin{equation}
    \label{eq:diff-delta}
    \begin{split}
      \Delta_n &= \frac{1}{n+1} \inf_{f \in \F} \E \bigg[ 
      \sum_{i=1}^{n+1} \ell_\phi (f, Z_i) - \sum_{i=1}^{n+1} \ell_\phi 
      (\wh f_{\phi, n+1}, Z_i) \bigg] \\
      & = \frac{1}{n+1} \E \bigg[ \sum_{i=1}^{n+1} \ell_\phi (f^*, Z_i) - \sum_{i=1}^{n+1} \ell_\phi (\wh f_{\phi, n+1}, Z_i) \bigg].      
    \end{split}
  \end{equation}
\end{lemma}

\begin{proof}
  This follows from the fact that inequality~\eqref{eq:proof-excessrisk-3} is an equality when $\wh g_n = \wt f_{\phi, n}$ (see Lemma~\ref{lem:minimax-logloss}).
\end{proof}

\subsection{Proofs for density estimation (Section~\ref{sec:density-estimation})}
\label{sub:proofs-density-estimation}

\begin{proof}[Proof of Proposition~\ref{prop:multinomial}]
  Since the MLE $\wh f_n$ writes $\wh f_n (y) = N_n(y) / n$, we have for every $y \in \Y$:
  \begin{equation}
    \label{eq:proof-multinomial-1}
    \wh f_n^{y} (y) = \frac{N_n(y) + 1}{n + 1}
    \propto N_n (y) + 1
    \, ,
  \end{equation}
  so that, since $\sum_{y \in \Y} N_n (y) = n$,
  \begin{equation}
    \label{eq:proof-multinomial-2}
    \sum_{y \in \Y} \wh f_n^{y} (y) = \frac{n + d}{n + 1}
    \, .
  \end{equation}
  It proves that the SMP $\wt f_n$~\eqref{eq:estimator-log} is the Laplace estimator~\eqref{eq:def-laplace} and that the excess risk bound~\eqref{eq:excess-risk-log} becomes $\E [\excessrisk(\wt f_n)] \leq \log \frac{n+d}{n+1} \leq \frac{d-1}{n}$ (since $\log (1+u) \leq u$ for $u \geq 0$).
\end{proof}

\begin{proof}[Proof of Proposition~\ref{prop:gaussian-location}]
  First, let us prove that a risk minimizer $f_{\theta^*, \Sigma} \in \F$ exists if and only if $\E \| Y \| < + \infty$ and that $\theta^* = \E [Y]$ in this case.
  Let $\mu$ be the distribution $\gaussdist(0, \Sigma)$, and define the log loss with respect to $\mu$.
  Then, for every $\theta, y \in \R^d$, $\ell(f_{\theta, \Sigma}, y) = - \langle \Sigma^{-1} \theta, y \rangle + \frac{1}{2} \| \theta \|_{\Sigma^{-1}}^2$.  
  Assume that there exists $\theta^* \in \R^d$ such that $\E [\ell (f_{\theta^*+\theta, \Sigma}, Y) - \ell (f_{\theta^*, \Sigma}, Y)]$ is well-defined and in $[0, +\infty]$ for every $\theta \in \R^d$.
  This implies that $\E [ (\ell (f_{\theta^*+\theta, \Sigma}, Y) - \ell (f_{\theta^*, \Sigma}, Y))_{-} ] < + \infty$, and hence that $\E [(\langle \Sigma^{-1} \theta, Y \rangle)_{-}] < + \infty$.
  Taking $\theta = \pm \Sigma e_j$ for $1\leq j \leq d$ (where $(e_j)_{1\leq j \leq d}$ is the canonical basis of $\R^d$), this implies that $\E |Y_j| < + \infty$, and hence that $\E \| Y \| \leq \E \| Y \|_1 = \sum_{j=1}^d \E [ |Y_j| ] < +\infty$.
  Conversely, if $\E \|Y\| < + \infty$, so that $\E [Y] \in \R^d$ exists, then for every $\theta \in \R^d$, $R (f_{\theta, \Sigma}) = \E [ \ell(f_{\theta, \Sigma}, Y) ] = - \langle \Sigma^{-1} \theta, \E[Y] \rangle + \frac{1}{2} \theta^\top \Sigma^{-1} \theta$, which is minimized by $\theta^* = \E[Y]$.

  We now proceed to determine the SMP and establish the excess risk bound~\eqref{eq:excessrisk-gaussian}.  
  The MLE is $f_{\bar Y_n, \Sigma} = \gaussdist(\bar{Y}_n, \Sigma)$, so that for $y \in \R^d$, $\wh f_n^y = f_{\wh \theta_n^y, \Sigma}$ with $\wh \theta_n^y = \frac{n \bar{Y}_n + y}{n + 1}$.
  Since $y - \wh \theta_n^y = \frac{n}{n+1} (y - \bar{Y}_n)$, we have, considering densities with respect to the measure $(2\pi)^{-d/2} \di y$:
  \begin{align}
    \label{eq:proof-gaussian-1}
    f_{\wh \theta_n^y} (y)
    &= (\det \Sigma)^{-1/2} \exp \Big( - \frac{1}{2}
      \big\| y - \wh \theta_n^y \big\|_{\Sigma^{-1}}^2
      \Big) \nonumber \\
    &= (\det \Sigma)^{-1/2} \exp \Big( - \frac{1}{2} \Big( \frac{n}{n+1} \Big)^2
      \big\| y - \bar Y_n \big\|_{\Sigma^{-1}}^2
    \Big) \nonumber \\
    &= (\det \Sigma)^{-1/2} \det ((1+1/n)^{2}\Sigma)^{1/2} f_{\bar{Y}_n, (1 + 1/n)^2 \Sigma} (y) \nonumber \\
    &= \Big( 1 + \frac{1}{n} \Big)^d f_{\bar{Y}_n, (1 + 1/n)^2 \Sigma} (y) \, ,
  \end{align}
  so that (after normalization) $\wt f_n = \gaussdist(\bar{Y}_n, (1+1/n)^2 \Sigma)$ and
  \begin{equation}
    \label{eq:proof-gaussian-2}
    \int_{\R^d} f_{\wh \theta_n^y} (y) (2\pi)^{-d/2} \di y
    = \int_{\R^d} \Big( 1 + \frac{1}{n} \Big)^d f_{\bar{Y}_n, (1 + 1/n)^2 \Sigma} (y) (2\pi)^{-d/2} \di y
    = \Big( 1 + \frac{1}{n} \Big)^d \, ,
  \end{equation}
  which yields the excess risk bound~\eqref{eq:excessrisk-gaussian} using Theorem~\ref{thm:excess-risk-log}.

  Now, assume that the model is well-specified, namely $Y \sim \gaussdist (\theta^*, \Sigma)$ for some $\theta^* \in \R^d$.
  Using Lemma~\ref{lem:exact-excessrisk-smp}, we have
  \begin{equation*}
    \E \big[ \excessrisk (\wt f_n) \big] 
    = \E \Big[ \log \Big( \int_{\R^d} f_{\wh \theta_n^y} (y) 
    (2\pi)^{-d/2} \di y \Big) \Big] - \Delta_n
    = d \log \Big( 1 + \frac{1}{n} \Big) - \Delta_n \, ,
  \end{equation*}
  where $\Delta_n$ is defined as in~\eqref{eq:proof-excessrisk-diff}, \ie
  \begin{align*}
    \Delta_n
    &= \frac{1}{n+1} \E \bigg[ \sum_{i=1}^{n+1} \ell (f_{\theta^*,\Sigma}, Y_i) - \inf_{\theta \in \R^d} \sum_{i=1}^{n+1} \ell (f_{\theta,\Sigma}, Y_i) \bigg] \\    
    &= \frac 12 \E \bigg[ \frac{1}{n+1} \sum_{i=1}^{n+1} \big\| Y_i - \theta^* \big\|_{\Sigma^{-1}}^2 - \frac{1}{n+1} \sum_{i=1}^{n+1} \big\| \bar{Y}_{n+1} - Y_i \big\|_{\Sigma^{-1}}^2 \bigg] \\    
    &= \frac 12 \E \big[ \| \bar{Y}_{n+1} - \theta^* \|_{\Sigma^{-1}}^2 \big] \\
    &= \frac{1}{2} \tr \Big( \Sigma^{-1} \E \big[ (\bar{Y}_{n+1} - \theta^*) (\bar{Y}_{n+1} - \theta^*)^\top \big] \Big) \\
    &= \frac{1}{2} \tr \Big( \Sigma^{-1} \times \frac{1}{n+1} \Sigma \Big) 
    = \frac{d}{2 (n+1)}      
  \end{align*}
  where we used the fact that $\E [(Y -\theta^*) (Y -\theta^*)^\top ] = \Sigma$.
  It follows that $\E \big[ \excessrisk (\wt f_n) \big] = d \log \left( 1 + 1/n \right) - d/(2n) \leq d/(2n)$, which completes the proof of Proposition~\ref{prop:gaussian-location}.
\end{proof}

\begin{proof}[Proof of Theorem~\ref{thm:gauss-minimax}]
  Define the densities and the log-loss with respect to the measure $(2\pi)^{-d/2} \di y$ on $\R^d$.
  For every $\sigma^2 > 0$, $\theta \in \R$ and $y \in \R^d$, we have
  \begin{equation*}
    \ell (f_{\theta, \sigma^2 \Sigma}, y)
    = - \log f_{\theta, \sigma^2 \Sigma} (y) 
    = \frac{d}{2} \log \sigma^2 + \frac{1}{2} \log \det (\Sigma) + \frac{1}{2 \sigma^2} \big\| y - \theta \big\|_{\Sigma^{-1}}^2
  \end{equation*}
  so that, denoting $\theta^* = \E [Y ]$
  and $\Sigma_Y := \E [(Y - \theta^*) (Y - \theta^*)^\top]$, we obtain
  \begin{align*}
    &R (f_{\theta, \sigma^2 \Sigma}) -  \frac{1}{2} \log \det (\Sigma)
    = \frac{d}{2} \log \sigma^2 + \frac{1}{2 \sigma^2} \E \big[ \big\| Y - \theta \big\|_{\Sigma^{-1}}^2 \big] \\
    &= \frac{d}{2} \log \sigma^2  + \frac{1}{2 \sigma^2} \big\| \theta - \theta^* \big\|_{\Sigma^{-1}}^2 + \frac{1}{2 \sigma^2} \E\, \tr \big( \Sigma^{-1} (Y - \theta^*) (Y - \theta^*)^\top \big) \\
    &= \frac{d}{2} \log \sigma^2  + \frac{1}{2 \sigma^2} \big\| \theta - \theta^* \big\|_{\Sigma^{-1}}^2 + \frac{1}{2 \sigma^2} \tr \big( \Sigma^{-1} \Sigma_Y \big)
  \end{align*}  
  so that
  \begin{align}
    \label{eq:proof-gauss-minimax-2}
    \excessrisk (f_{\theta, \sigma^2 \Sigma})
    &= R (f_{\theta, \sigma^2 \Sigma}) - R (f_{\theta^*, \Sigma}) \nonumber \\
    &= \frac{d}{2} \log \sigma^2 + \frac{1}{2 \sigma^2} \big\| \theta - \theta^* \big\|_{\Sigma^{-1}}^2 + \frac{1}{2} \Big( \frac{1}{\sigma^2} - 1 \Big) \tr \big( \Sigma^{-1} \Sigma_Y 
    \big) \, .
  \end{align}
  Now, since
  \begin{equation*}
    \E \big[ \big\| \bar{Y}_n - \theta^* \big\|_{\Sigma^{-1}}^2 \big]
    = \tr \Big( \Sigma^{-1} \E \big[ (\bar{Y}_n - \theta^*) (\bar{Y}_n - \theta^*)^\top \big] \Big)
    = \frac{\tr (\Sigma^{-1} \Sigma_Y)}{n}
    \, ,
  \end{equation*}
  equation~\eqref{eq:proof-gauss-minimax-2} implies that, for $\sigma^2 = 1 + 1/n$,
  \begin{equation}
    \label{eq:proof-gauss-minimax-3}
    \E \big[ \excessrisk(f_{\bar{Y}_n, \sigma^2 \Sigma}) \big]
    = \frac{d}{2} \log \sigma^2 + \frac{1}{2} \Big[ \Big( 1 + \frac{1}{n} \Big) \frac{1}{\sigma^2} - 1 \Big] \tr (\Sigma^{-1} \Sigma_Y)
    = \frac{d}{2} \log \Big( 1 + \frac{1}{n} \Big)
    \, .
  \end{equation}
  In order to conclude that $\wh f_n = \gaussdist(\bar{Y}_n, (1+1/n) \Sigma)$, which has constant risk, achieves minimax excess risk over the class of distributions of $Y$ with finite variance, it suffices to note that $\wh f_n$ achieves minimax excess risk for $Y$ a Gaussian from $\{ \gaussdist (\theta^*, \Sigma) : \theta^* \in \R^d \}$ (\ie, in the well-specified case).
  Indeed, if $Y \sim \gaussdist(\theta^*, \Sigma)$, then $\excessrisk(f) = \kll{\gaussdist(\theta^*, \Sigma)}{f}$ for every density $f$, and $\wh g_n$ achieves minimax KL-risk on the Gaussian location family~\cite{ng1980estimation,murray1977density}.
\end{proof}

\subsection{Proofs for the Gaussian linear model (Section~\ref{sec:cond-gauss})}
\label{sub:proof-cond-gauss-line}

\begin{proof}[Proof of Theorem~\ref{thm:smp-gaussian-linear}]
  Let us first recall that 
  $\F = \{ f_{\theta} (y \cond x) = \gaussdist (\langle \theta, x \rangle, 1) : \theta \in \R^d \}$ and that $\wh \Sigma_n = n^{-1} \sum_{i=1}^n X_i X_i^\top$ and $\wh S_n = n^{-1} \sum_{i=1}^n Y_i X_i$.
  The MLE is given by $\wh \theta_n = \wh \Sigma_n^{-1} \wh S_n$ and, for every $x \in \R^d$ and $y \in \R$, 
\begin{equation*}
  \wh \theta_n^{(x, y)} = (n \wh \Sigma_n + x x^\top)^{-1} 
  (n \wh S_n + y x).
\end{equation*}
Hence, we have
\begin{align*}
  y - \langle \wh \theta_n^{(x,y)}, x \rangle
  &= y - \big \langle (n \wh\Sigma_n + x x^\top )^{-1} ( n \wh S_n + y x ), x \big \rangle \\
  &= \big( 1 - \big \langle (n \wh\Sigma_n + x x^\top )^{-1} x, x \big 
  \rangle \big) y - \big \langle (n \wh\Sigma_n + x x^\top )^{-1} n \wh S_n, x \big \rangle \\
  &= \sigma_n(x)^{-1} ( y - \mu_n (x) ),
\end{align*}
where we defined
\begin{equation*}
  \sigma_n (x) = \big( 1 - \big \langle (n \wh\Sigma_n + x x^\top )^{-1} x, x \big \rangle \big)^{-1} \quad \text{ and } \quad \mu_n (x)
  = \frac{\big \langle (n \wh \Sigma_n + x x^\top )^{-1} n \wh S_n, x \big \rangle}{1 - \big \langle (n \wh \Sigma_n + x x^\top)^{-1} x, x 
  \big \rangle}.
\end{equation*}
Note that both quantities are well-defined under since $\wh \Sigma_n$ is invertible almost surely by Assumption~\ref{ass:weak-X}.
Moreover, these quantities can be simplified thanks to the following lemma.
\begin{lemma}
  \label{lem:sherman}
  Assume that $S$ is a symmetric positive $d$-dimensional matrix and that $v\in \R^d$. 
  Then\textup, one has
  \begin{equation}
    \label{eq:sherman1}
    \big( 1 - \langle (S + v v^\top)^{-1} v, v\rangle \big)^{-1}
    = 1 + \langle S^{-1} v, v\rangle,
  \end{equation}
  and, for any $u \in \R^d$,
  \begin{equation}
    \label{eq:sherman2}
    \frac{\big \langle (S + v v^\top)^{-1} S u, v \big \rangle}{1 - \big \langle  (S + v v^\top)^{-1} v, v \big \rangle} = \langle u, v \rangle
    \, .
  \end{equation}  
\end{lemma}
The proof of Lemma~\ref{lem:sherman} is given below.
It also follows from the Sherman-Morrison formula.
Using~\eqref{eq:sherman1} with $S = n \wh \Sigma_n$ and $v = x$ leads to
\begin{equation*}
  \sigma_n (x) = 1 + \big \langle (n \wh \Sigma_n)^{-1} x, x \big 
  \rangle
\end{equation*}
while the fact that $\wh S_n = \wh \Sigma_n \wh \theta_n$ together with~\eqref{eq:sherman2} for $S = n \wh \Sigma_n$, $v = x$ and $u = \wh \theta_n$ leads to
\begin{equation*}
  \mu_n (x)
  = \frac{\big \langle (n \wh \Sigma_n + x x^\top )^{-1} n \wh S_n, x \big \rangle}{1 - \big \langle (n \wh \Sigma_n + x x^\top)^{-1} x, x 
  \big \rangle} = \frac{\big \langle (n \wh \Sigma_n + x x^\top )^{-1} n \wh \Sigma_n \wh \theta_n, x \big \rangle}{1 - \big \langle (n \wh \Sigma_n + x x^\top)^{-1} x, x 
  \big \rangle} = \langle \wh \theta_n, x \rangle.
\end{equation*}
Consider the dominating measure $\mu(\di y) = (2\pi)^{-1/2} \di y$ 
on $\R$. 
The computations above entail that for every $y \in \R$, we have
\begin{equation*}
  f_{\wh \theta_n^{(x,y)}} (y \cond x)
  = \frac{1}{\sqrt{2 \pi}} \exp \Big( - \frac 12 \big(y - \langle \wh \theta_n^{(x,y)}, x \rangle \big)^2 \Big)
  = \frac{1}{\sqrt{2 \pi}} \exp \Big(- 
  \frac{1}{2 \sigma^2_n (x)} \big( y - \mu_n(x) \big)^2 \Big).
\end{equation*}
Note that
\begin{equation*}
  \int_{\R} f_{\wh \theta_n^{(x,y)}} (y \cond x) \mu(\di y) = \sigma_n (x),
\end{equation*}
which shows after normalization~\eqref{eq:estimator-log} that the SMP is given by
\begin{equation}
  \label{eq:predictor-linear-gaussian}
  \wt f_n (y \cond x) = \gaussdist (  \mu_n (x), \sigma_n^2 (x))
\end{equation}
and that its excess risk writes
\begin{equation}
  \label{eq:excessrisk-linear-gaussian}
  \E \big[ \excessrisk (\wt f_n) \big]
  \leq \E \big[ \log \sigma_n (X) \big]
  = \E \Big[ - \log \Big( 1 - \big \langle ( n \wh\Sigma_n + X X^\top )^{-1} X, X \big \rangle \Big) \Big]
  \, .
\end{equation}
This proves the first inequality in~\eqref{eq:excessrisk-smp-linear}.
Let us prove now the second inequality in~\eqref{eq:excessrisk-smp-linear}.
Let us recall that the covariance $\Sigma$ and rescaled design $\wt X, \wt X_i$ and rescaled covariance $\wt \Sigma_n$ are given by~\eqref{eq:def-empirical-covariance} and~\eqref{eq:def-rescaled_covariance}.
We have
\begin{equation}
  \label{eq:bound-decorr}
  \begin{split}
    \big \langle ( n \wh \Sigma_n + X X^\top )^{-1} X, X \big\rangle
    & = \big \langle \Sigma^{1/2} ( n \wh \Sigma_n + X X^\top )^{-1} 
    \Sigma^{1/2} \Sigma^{-1/2} X, \Sigma^{-1/2} X \big\rangle \\
    &= \big\langle ( n \wt \Sigma_n + \wt X \wt X^\top )^{-1} \wt X, \wt X \big\rangle,
  \end{split}
\end{equation}
hence, combining~\eqref{eq:excessrisk-linear-gaussian},~\eqref{eq:bound-decorr} and~\eqref{eq:sherman1}, we have
\begin{equation*}
    \E \big[ \excessrisk (\wt f_n) \big] \leq \E \Big[ 
    - \log \Big( 1 - \big \langle (n \wt \Sigma_n + 
    \wt X \wt X^\top)^{-1} \wt X, \wt X \big \rangle \Big) \Big] 
  = \E \Big[ \log \Big( 1 + \big \langle (n \wt \Sigma_n)^{-1} \wt X, \wt X \big \rangle \Big) \Big],
\end{equation*}
which leads, using Jensen's inequality, together with $\E [\wt X \wt X^\top] = \id$ and the fact that $\wt \Sigma_n$ and $\wt X$ are independent, to
\begin{align*}
  \E \big[ \excessrisk (\wt f_n) \big]
  \leq \log \Big( 1 + \frac{1}{n} \E \big[ \tr \big( \wt \Sigma_n^{-1} \wt X \wt X^\top \big) \big] \Big) 
  &= \log \Big( 1 + \frac{1}{n} \tr \big \{ \E [ \wt \Sigma_n^{-1} ] \E [ \wt X \wt X^\top ] \big \} \Big) \\
  &= \log \Big( 1 + \frac{1}{n} \E [ \tr ( \wt \Sigma_n^{-1} ) ] \Big).
\end{align*}
This concludes the proof of Theorem~\ref{thm:smp-gaussian-linear}.
\end{proof}

\begin{proof}[Proof of Lemma~\ref{lem:sherman}]
  First,~\eqref{eq:sherman2} clearly holds if $v = 0$.
  Now, for $u, v \in \R^d$, $v \neq 0$:
  \begin{align}
    \big \langle (S + v v^\top)^{-1} S u, v \big \rangle
    &= \big \langle (S + v v^\top)^{-1} (S + v v^\top - v v^\top) u, v \big \rangle \nonumber \\
    &= \big \langle (\id - (S + v v^\top)^{-1} v v^\top) u, v \big \rangle \nonumber \\
    &= \langle u, v\rangle \big( 1 - \langle (S + v v^\top)^{-1} v, v \rangle \big) \label{eq:proof-sm}
      \, .
  \end{align}
  Letting $u = S^{-1} v$ in~\eqref{eq:proof-sm}, the left-hand side is $ \langle (S + v v^\top)^{-1} v, v \rangle > 0$ (since $S + v v^\top \mgeq S$ is positive, and $v \neq 0$) so that the right-hand side is positive and thus $1 - \langle (S + v v^\top)^{-1} v, v \big \rangle > 0$.
  Dividing both sides of~\eqref{eq:proof-sm} by this quantity establishes~\eqref{eq:sherman2}, which implies~\eqref{eq:sherman1} by taking $u = S^{-1} v$.
\end{proof}

\begin{proof}[Proof of Theorem~\ref{thm:smp-linear-ridge} and Proposition~\ref{prop:smp-ridge-log-norm}]
  Let us recall that we consider the family $\F = \{ f_{\theta}(\cdot \cond x) = \gaussdist (\langle \theta, x \rangle, \sigma^2) : \theta \in \R^d \}$, together with the Ridge penalization $\phi (\theta) = {\lambda} \| \theta \|^2 / 2$ for some $\lambda > 0$.
  Let 
\begin{equation*}
  \wh \theta_{\lambda, n}
  := \argmin_{\theta \in \R^d} \bigg\{ \frac{1}{n} \sum_{i=1}^n \ell (f_\theta, (X_i, Y_i)) + \frac{\lambda}{2} \| \theta \|^2 \bigg\}
  = (\wh \Sigma_n + \lambda \id)^{-1} \wh S_n,
\end{equation*}
denote the Ridge estimator, where
$\wh \Sigma_n$ and $\wh S_n$ are the same as in the proof of Theorem~\ref{thm:smp-gaussian-linear}.
Defining 
\begin{equation*}
  \wh \Sigma_\lambda^x = n \wh \Sigma_n + x x^\top + \lambda (n+1) \id \quad \text{ and } \quad \wh K_\lambda^x = ( \wh \Sigma_\lambda^x )^{-1},
\end{equation*}
we have
\begin{equation*}
  \wh \theta_{\lambda, n}^{(x,y)} = \big(n \wh \Sigma_n + x x^\top + \lambda (n+1) \id \big)^{-1} (n \wh S_n + y x) = \wh K_\lambda^x (n \wh S_n + yx)
\end{equation*}
for any $y \in \R$ and $x \in \R^d$.
Note that we have
\begin{equation}
  \label{eq:thm5proofresidual}
  y - \langle \wh \theta_{\lambda, n}^{(x,y)}, x \rangle
  = y - \big \langle \wh K_\lambda^x
    ( n \wh S_n + y x), x \big \rangle
  = \big( 1 - \| x \|_{\wh K_\lambda^x}^2 \big) y - \langle n \wh S_n, x \rangle_{\wh K_\lambda^x}
\end{equation}
and that
\begin{align*}
  \lambda \| \wh \theta_{\lambda, n}^{(x,y)} \|^2 
  &= \lambda \| \wh K_\lambda^x (n \wh S_n + y x) \|^2
  = \lambda \| n \wh S_n + y x \|_{(\wh K_\lambda^x)^2}^2 \\
  &= y^2 \lambda \| x \|_{(\wh K_\lambda^x)^2}^2 + 2 y \lambda \langle n \wh S_n, x
  \rangle_{(\wh K_\lambda^x)^2} +  \lambda  \| n \wh S_n  \|_{(\wh K_\lambda^x)^2}^2 \, .
\end{align*}
The SMP is given in this setting by 
\begin{equation*}
  \wt f_{\lambda, n} (y | x) = \frac{f_{\wh \theta_{\lambda, n}^{(x,y)}} (y \cond x) e^{-\lambda \| \wh \theta_{\lambda, n}^{(x,y)} \|^2 / 2 }}{\int_\R f_{\wh \theta_{\lambda, n}^{(x,y')}} (y' \cond x) e^{-\lambda \| \wh \theta_{\lambda, n}^{(x,y')} \|^2 / 2 } \mu(\di y')},
\end{equation*}
where $\mu(\di y) = (2 \pi)^{-1/ 2} \di y$, see~\eqref{eq:estimator-log}, and where
\begin{equation*}
  f_{\wh \theta_{\lambda, n}^{(x,y)}} (y \cond x) e^{-\lambda \| \wh \theta_{\lambda, n}^{(x,y)} \|^2 / 2 } = \exp \bigg( -\frac 12 \Big\{ \big ( y - \langle \wh \theta_{\lambda, n}^{(x,y)}, x \rangle \big )^2 + \lambda \| \wh \theta_{\lambda, n}^{(x,y)} \|^2 \Big\} \bigg).
\end{equation*}
Now, the equality~\eqref{eq:thm5proofresidual} gives, after a straightforward computation,
\begin{equation*}
  \big ( y - \langle \wh \theta_{\lambda, n}^{(x,y)}, x \rangle \big )^2 + \lambda \| \wh \theta_{\lambda, n}^{(x,y)} \|^2 = \frac{1}{\sigma_\lambda(x)^2} \big( y - \mu_\lambda(x) \big)^2 + C,
\end{equation*}
where $C$ is a quantity that does not depend on $y$ and where we introduced, respectively, 
\begin{align*}
  \sigma_\lambda(x)^2
  &=
  \Big( {( 1 - \| x \|_{\wh K_\lambda^x}^2 )^2 + \lambda \| x \|_{(\wh K_\lambda^x)^2}^2 } \Big)^{-1} \\
  \mu_\lambda(x) 
  &=  \frac{( 1 - \| x \|_{\wh K_\lambda^x}^2 ) \langle n \wh S_n, x \rangle_{\wh K_\lambda^x}  - \lambda \langle n \wh S_n, x \rangle_{(\wh K_\lambda^x)^2} }{( 1 - \| x \|_{\wh K_\lambda^x}^2 )^2 
    + \lambda \| x \|_{(\wh K_\lambda^x)^2}^2}
  \, .
\end{align*}
This entails that the SMP is given by
\begin{equation}
  \wt f_{\lambda, n} (\cdot | x) = \gaussdist \big( \mu_\lambda(x), \sigma_\lambda(x)^2 \big)
  \, .
\end{equation}
By definition of $\wh \theta_{\lambda, n}$ we have
\begin{equation*}
  n \wh S_n = \big(n \wh \Sigma_n + \lambda (n + 1) \id \big) \wh \theta_{\lambda', n}
\end{equation*}
where $\lambda' = (n + 1) \lambda / n$, so that for  $\alpha \in \{ 1, 2 \}$
we have
\begin{align*}
  \langle n \wh S_n, x \rangle_{(\wh K_\lambda^x)^\alpha} 
  &= \big \langle \big( n \wh \Sigma_n + x x^\top + \lambda (n+1) \id \big)^\alpha
   n \wh S_n, x \big \rangle \\
   & = \big \langle \big( n \wh \Sigma_n + x x^\top + \lambda (n+1) \id \big)^\alpha
   (n \wh \Sigma_n + \lambda (n + 1) \id + x x^\top - x x^\top) 
   \wh \theta_{\lambda', n}, x \big \rangle \\
   &=  \langle \wh \theta_{\lambda', n}, x \rangle_{(\wh K_\lambda^x)^{\alpha - 1}} 
    -  \langle \wh \theta_{\lambda', n}, x \rangle \, 
    \| x \|_{(\wh K_\lambda^x)^\alpha}^2,
\end{align*}
namely
\begin{equation*}
  \langle n \wh S_n, x \rangle_{\wh K_\lambda^x}
  = \big(1  -  \| x \|_{\wh K_\lambda^x}^2 \big) \langle \wh \theta_{\lambda', n}, x \rangle
  \quad \text{ and } \quad \langle n \wh S_n, x \rangle_{(\wh K_\lambda^x)^2} =  \langle \wh \theta_{\lambda', n}, x \rangle_{\wh K_\lambda^x} -  \langle \wh \theta_{\lambda', n}, x \rangle \, \| x \|_{(\wh K_\lambda^x)^2}^2.
\end{equation*}
This allows, after straightforward computations, to express $\mu_{\lambda}(x)$ as a function of $\wh \theta_{\lambda', n}$ as follows:
\begin{equation*}
  \mu_\lambda(x) = \langle \wh \theta_{\lambda', n}, x \rangle - \lambda 
  \sigma_\lambda(x)^2  \langle \wh \theta_{\lambda', n}, x \rangle_{\wh K_\lambda^x}.
\end{equation*}
We know from Theorem~\ref{thm:excess-risk-log} that the penalized excess risk of SMP satisfies
\begin{align*}  
  \E \big[ \excessrisk_\lambda (\wt f_{\lambda, n}) \big]
  &\leq \E_{Z_1^n,X} \Big[ \log \Big( \int_{\R} f_{\wh \theta_{\lambda, n}^{(X,y)}} (y \cond X) e^{- \lambda \| \wh \theta_{\lambda, n}^{(X,y)} \|^2 /2} \mu(\di y) \Big) \Big] \\
  & \leq \E_{Z_1^n,X} \Big[ \log \Big( \int_{\R} f_{\wh \theta_{\lambda, n}^{(X,y)}} (y \cond X)  \mu(\di y) \Big) \Big].
\end{align*}
We know from the computations above that
\begin{equation*}
  \big( y - \langle \wh \theta_{\lambda, n}^{(x,y)}, x \rangle \big)^2
  = \big( 1 - \| x \|_{\wh K_\lambda^x}^2 \big)^2 \big( y - \langle \wh \theta_{\lambda', n}, x \rangle \big)^2,
\end{equation*}
so that, after integrating with respect to $y$,
\begin{equation}
  \label{eq:proof-excessrisk-smp-ridge}
  \E \big[ \excessrisk_\lambda (\wt f_{\lambda, n}) \big]
  \leq \E_{X_1^n,X} \bigg[ 
  \log \Big( \frac{1}{ 1 - \| X \|_{\wh K_\lambda^X}^2} \Big) \bigg]
  = \E_{X_1^n,X} \Big[ - \log \big( {1 - \langle (\wh \Sigma_\lambda^X)^{-1} X, X \rangle} \big) \Big].
\end{equation}
Note that, by the identity~\eqref{eq:sherman1} from Lemma~\ref{lem:sherman}, and since $\| X \| \leq R$ almost surely (Assumption~\ref{ass:bounded-covariates}) we have
\begin{equation}
  \label{eq:bound-leverage-ridge}
  \langle (\wh \Sigma_\lambda^X)^{-1} X, X \rangle
  = \frac{\langle (n \wh \Sigma_n + \lambda (n+1) \id)^{-1} X, X \rangle}{1 + \langle (n \wh \Sigma_n + \lambda (n+1) \id)^{-1} X, X \rangle}
  \leq \frac{R^2 / \big( \lambda (n+1) \big)}{1 + R^2 / \big( \lambda (n+1) \big)}
  \, .
\end{equation}
In addition, the function $g (u) = - \log (1 - u) / u$ defined on $(0, 1)$ is nondecreasing, since its derivative writes:
\begin{equation*}
  g' (u)
  = \frac{1}{u^2} \left[ \frac{u}{1 - u} - \log \left( 1 + \frac{u}{1 - u} \right) \right]
  \geq 0
  \, ,
\end{equation*}
where we used the inequality $\log (1 + v) \leq v$ for $v \geq 0$.
Combining this fact with~\eqref{eq:bound-leverage-ridge} shows that 
\begin{equation}
  \label{eq:proof-smp-linear-ridge-linearize-g}
  - \log \big( 1 - \langle (\wh \Sigma_\lambda^X)^{-1} X, X \rangle \big)
  \leq g \bigg( \frac{R^2 / \big( \lambda (n+1) \big)}{1 + R^2 / \big( \lambda (n+1) \big)} \bigg) \cdot \langle (\wh \Sigma_\lambda^X)^{-1} X, X \rangle
  \, .
\end{equation}
Next, by exchangeability of $(X_1, \dots, X_n, X)$, we have
\begin{align}
  \label{eq:expect-leverage-ridge}
  &\E \big[ \langle (\wh \Sigma_\lambda^X)^{-1} X, X \rangle \big]
  = \frac{1}{n+1} \, \E \bigg[ \sum_{i=1}^n \langle (\wh \Sigma_\lambda^X)^{-1} X_i, X_i \rangle + \langle (\wh \Sigma_\lambda^X)^{-1} X, X \rangle \bigg] \nonumber \\
  &= \frac{1}{n+1} \, \E \bigg[ \tr \bigg\{ \bigg( \sum_{i=1}^n X_i X_i^\top + X X^\top + \lambda (n+1) \id \bigg)^{-1} \bigg( \sum_{i=1}^n X_i X_i^\top + X X^\top \bigg) \bigg\} \bigg]
    .
\end{align}
In addition, the function $A \mapsto \tr ( (A + \id)^{-1} A)$ is concave on positive matrices.
Indeed, it writes $d - \tr [ (A + \id)^{-1} ]$, and $A \mapsto \tr (A^{-1})$ is convex on positive matrices since $x \mapsto x^{-1}$ is convex on $\R_+^*$, by a general result on the convexity of trace functionals, see e.g.~\cite{bhatia2009pdmatrices,boyd2004convex}.
Hence, applying Jensen's inequality to~\eqref{eq:expect-leverage-ridge} and using the fact that
\begin{equation*}
  \E \bigg[ \sum_{i=1}^n X_i X_i^\top + X X^\top \bigg]
  = (n+1) \Sigma
  \, ,
\end{equation*}
we obtain:
\begin{equation}
  \label{eq:expect-leverage-ridge-2}
  \E \big[ \langle (\wh \Sigma_\lambda^X)^{-1} X, X \rangle \big]
  \leq \frac{\dflambda{\Sigma}}{n+1}
  \, .
\end{equation}
Finally, combining the bounds~\eqref{eq:proof-excessrisk-smp-ridge},~\eqref{eq:proof-smp-linear-ridge-linearize-g} and~\eqref{eq:expect-leverage-ridge-2} yields:
\begin{equation}
  \label{eq:proof-excessrisk-smp-ridge-general}
  \E \big[ \excessrisk_\lambda (\wt f_{\lambda, n}) \big]
  \leq g \bigg( \frac{R^2 / \big( \lambda (n+1) \big)}{1 + R^2 / \big( \lambda (n+1) \big)} \bigg) \cdot \frac{\dflambda{\Sigma}}{n+1}
  \, .
\end{equation}

\paragraph{Nonparametric rates (Theorem~\ref{thm:smp-linear-ridge}).}

Assume that $\lambda (n+1) \geq 2 R^2$.
The quantity inside $g (\cdot)$ in~\eqref{eq:proof-excessrisk-smp-ridge-general} is then bounded by $(1/2) / (1 + 1/2) = 1/3$, and since $g(1/3) = 3 \log (3/2) \leq 1.25$,~\eqref{eq:proof-excessrisk-smp-ridge-general} becomes, by definition of $\excessrisk_\lambda$:
\begin{equation}
  \label{eq:proof-excessrisk-smp-ridge-nonparametric}
  \E \big[ R (\wt f_{\lambda, n}) \big] - \inf_{\theta \in \R^d} \bigg\{ R (f_\theta) + \frac{\lambda}{2} \| \theta \|^2  \bigg\}
  \leq 1.25 \cdot \frac{\dflambda{\Sigma}}{n+1}
  \, .
\end{equation}
which is precisely the announced bound~\eqref{eq:excessrisk-smp-linear-ridge}.

\paragraph{Finite-dimensional case: improved dependence on the norm (Proposition~\ref{prop:smp-ridge-log-norm}).}

Now, let $\lambda = d / \big( B^2 (n+1) \big)$ for some $B > 0$ (which will be a bound on the norm of the comparison parameter $\theta$).
Then, $R^2 / \big( \lambda (n+1) \big) = B^2 R^2 / d$.
Now, note that for every $v > 0$
\begin{equation*}
  g \bigg( \frac{v}{1 + v} \bigg)
  = \frac{- \log \big( 1 - v / (1 + v) \big)}{v / (1 + v)}
  = \frac{(1 + v) \log (1 + v)}{v}
  \, .
\end{equation*}
In addition, if $v \leq 1$, then $(1 + v) \log (1 + v) / v \leq 1 + v \leq 2$.
On the other hand, if $v \geq 1$, then $(1 + v)/ v \leq 2$; it follows that for every $v > 0$:
\begin{equation}
  \label{eq:proof-logB-linear-g}
  g \bigg( \frac{v}{1 + v} \bigg)
  \leq 2 \log (e + v)
  \leq 2 \log (4 + 4 \sqrt{v} + v)
  = 4 \log (2 + \sqrt{v})
  \, .
\end{equation}
Now, the excess risk bound~\eqref{eq:proof-excessrisk-smp-ridge-general} implies that, for every $\theta \in \R^d$ such that $\| \theta \| \leq B$,
\begin{align}  
  \E \big[ R (\wt f_{\lambda, n}) \big] - R (f_\theta)
  &\leq g \bigg( \frac{B^2 R^2 / d}{1 + B^2 R^2 / d} \bigg) \cdot \frac{\dflambda{\Sigma}}{n + 1} + \frac{\lambda}{2} \| \theta \|^2 \nonumber \\
  &\leq 4 \log \left( 2 + \frac{B R}{\sqrt{d}} \right) \times \frac{d}{n+1} + \frac{d}{B^2 (n+1)} \times \frac{B^2}{2} \label{eq:proof-logB-linear-1} \\
  &= \frac{d}{n+1} \left\{ 4 \log \left( 2 + \frac{B R}{\sqrt{d}} \right) + \frac{1}{2} \right\} \nonumber \\
  &\leq \frac{5 d \log \big( 2 + {B R}/{\sqrt{d}} \big)}{n + 1} \label{eq:proof-logB-linear-final}
\end{align}
where inequality~\eqref{eq:proof-logB-linear-1} uses the bound~\eqref{eq:proof-logB-linear-g} with $v = B^2 R^2 / d$, the bound $\dflambda{\Sigma} \leq d$~\eqref{eq:dflambda-smaller-d} and the fact that $\| \theta \| \leq B$, while inequality~\eqref{eq:proof-logB-linear-final} uses the fact that $1/2 \leq \log 2$.
\end{proof}

\subsection{Proofs for logistic regression (Section~\ref{sec:logistic})}
\label{sec:proofs-logistic}

\begin{proof}[Proof of Proposition~\ref{prop:logistic}]
  Let us first discuss the properties of predictions produced by the SMP, and compare it to the MLE.
  First, if the points $Z_1, \dots, Z_n$ do not lie within a half-space, the MLE is uniquely determined and belongs to $\R^d$; in addition, for any $x \in \R^d$ and $y \in \{-1, 1\}$, $Z_1, \dots, Z_n, -y x$ are not separated either, so $\wh \theta_n^{(x,y)} \in \R^d$ is also well-defined and unique, and so is the prediction $\wt f_n (1 \cond x) \in (0, 1)$.

Let $\Lambda_n = \{ \sum_{1 \leq i \leq n} \lambda_i Z_i : \lambda_i \in \R^+, 1\leq i \leq n \}$ denote the convex cone generated by $Z_1, \dots, Z_n$.
Assume that $\Lambda_n \cap (- \Lambda_n) = \{ 0 \}$ and that all $Z_i$ are distinct from $0$.
Then, convex separation implies that there exists $\theta \in \R^d$ such that $\langle \theta, z\rangle < 0$ for all $z \in \Lambda_n \setminus \{ 0 \}$, so that the $Z_i$ lie within a strict half-space: $\langle \theta, Z_i\rangle < 0$ for all $i$.
Hence, any MLE $f_{\wh \theta_n}$ in $\overline{\F}$ belongs to $\overline{\F} \setminus \F$, and corresponds to a separating hyperplane $(+ \infty, \wh \theta_n)$ for some $\wh \theta_n \in S^{d-1}$ (such that $\langle \wh \theta_n, z\rangle < 0$ for all $z \in \Lambda_n \setminus \{0\}$).
Its predictions $f_{\wh \theta_n} (1 \cond x)$ are as follows:
\begin{itemize}
\item If $x = 0$, then $f_{\wh \theta_n} (1 \cond x) = 1/2$. 
\item If $x \in \Lambda_n \setminus \{ 0 \}$, then $\langle \wh\theta_n, x\rangle < 0$ and thus $f_{\wh \theta_n} (1 \cond x) = 0$. Likewise, if $x \in (-\Lambda_n) \setminus \{ 0 \}$, then $f_{\wh \theta_n} (1 \cond x) = 1$;
\item If $x \in \R^d \setminus [ \Lambda_n \cup (- \Lambda_n) ]$, then both $x$ and $-x$ are linearly separated from $\Lambda_n$.
  Hence, one can choose $\wh \theta_n$ with $\langle \wh \theta_n, z\rangle < 0$ for $z \in \Lambda_n \setminus \{ 0 \}$ such that either $\langle \wh \theta_n, x\rangle > 0$ or $\langle \wh \theta_n, x\rangle < 0$ (or even $\langle \wh \theta_n, x\rangle = 0$).
  In other words, one can choose an MLE $\wh \theta_n$ such that $f_{\wh \theta_n} (1 \cond x)$ is either $1$, $0$ or $1/2$: the prediction of the MLE is ill-determined in this region, since it depends on the specific choice of the MLE.
\end{itemize}
By contrast, let us consider the prediction of the SMP $\wt f_n$.
Let $z = - y x \in \R^d \setminus \{ 0 \}$.
As before, if $z \in \R^d \setminus (- \Lambda_n)$, then there exists $\theta$ with $\langle \theta, z\rangle < 0$ and $\langle \theta, Z_i\rangle = - \langle \theta, -Z_i\rangle < 0$.
Hence, $f_{\wh \theta_n^{(x, y)}} (y \cond x) = 1$.
On the other hand, if $z \in (- \Lambda_n) \setminus \{ 0 \}$, then the dataset $Z_1, \dots, Z_n, z$ is not separated, so that $f_{\wh \theta_n^{(x, y)}} (y \cond x) \in (0, 1)$.
Hence, for $x \in \R^d$:
\begin{itemize}
\item If $x = 0$, then $\wt f_n (1 \cond x) = 1/2$.
\item If $x \in \Lambda_n$, then $-x \in (- \Lambda_n)$ so that $f_{\wh \theta_n^{(x, 1)}} (1 \cond x) \in (0, 1)$, while $x \in \R^d \setminus (-\Lambda_n)$ so that $f_{\wh \theta_n^{(x, -1)}} (-1 \cond x) = 1$; hence, $\wt f_n (1 \cond x) \in (0, 1/2)$.
  Likewise, if $x \in (-\Lambda_n)$, then $\wt f_n (1 \cond x) \in (1/2, 1)$.
\item If $x \in \R^d \setminus [\Lambda_n \cup (-\Lambda_n)]$, then $f_{\wh \theta_n^{(x, 1)}} (1 \cond x) = f_{\wh \theta_n^{(x, -1)}} (-1 \cond x) = 1$, so that $\wt f_n (1 \cond x) = 1/2$.
\end{itemize}
Finally, the excess risk bound~\eqref{eq:excessrisk-logistic-exact} is established in the proof of Theorem~\ref{thm:smp-linear-ridge} below, letting $\lambda = 0$.
\end{proof}

\begin{proof}[Proof of Theorem~\ref{thm:logistic-ridge-smp}]
  Let $(X, Y)$ be a test sample, and $Z = - Y X$.
  Since $\{ Z, - Z \} = \{ X, - X\}$, the excess risk bound~\eqref{eq:excess-risk-log} of the SMP $\wt f_{\lambda,n}$~\eqref{eq:smp-ridge-logistic} writes:
  \begin{align}
    &\E \big[ R (\wt f_{\lambda, n}) \big]
    - \inf_{\theta \in \R^d} \Big\{ R (f_\theta) + \frac{\lambda}{2} \| \theta \|^2 \Big\} \nonumber \\
    &\leq \E \left[ \log \left( \sigma (\langle \wh \theta_{\lambda, n}^{(X,1)}, X\rangle) \, e^{-\lambda \| \wh \theta_{\lambda, n}^{(X,1)} \|^2/2 } + \sigma ( -\langle \wh \theta_{\lambda, n}^{(X,-1)}, X\rangle) \, e^{-\lambda \| \wh \theta_{\lambda, n}^{(X,-1)} \|^2/2 } \right) \right] \nonumber \\
    &= \E \left[ \log \left( \sigma (\langle \wh \theta_{\lambda, n}^{-Z}, Z\rangle) \, e^{-\lambda \| \wh \theta_{\lambda, n}^{-Z} \|^2/2 } + \sigma ( -\langle \wh \theta_{\lambda, n}^{Z}, Z\rangle) \, e^{-\lambda \| \wh \theta_{\lambda, n}^{Z} \|^2/2 } \right) \right] \nonumber \\
    &\leq \E \left[ \log \left( 1 + \sigma (\langle \wh \theta_{\lambda, n}^{-Z}, Z\rangle) - \sigma ( \langle \wh \theta_{\lambda, n}^{Z}, Z\rangle) \right) \right] \label{eq:proof-logistic-ridge-1} \\
    &\leq \E \left[ \sigma (\langle \wh \theta_{\lambda, n}^{-Z}, Z\rangle) - \sigma ( \langle \wh \theta_{\lambda, n}^{Z}, Z\rangle) \right]
      \label{eq:proof-logistic-ridge-2}
  \end{align}
  where inequality~\eqref{eq:proof-logistic-ridge-1} is obtained by lower-bounding $e^{-\lambda \| \cdot \|^2 / 2} \leq 1$ and using the identity $\sigma (-u) = 1 - \sigma (u)$.
  Now, defining for $\theta \in \R^d$
  \begin{equation*}
    \wh R_{\lambda, n}^Z (\theta)
    := \frac{1}{n+1} \bigg\{ \sum_{i=1}^n \ell (\langle \theta, Z_i\rangle) + \ell (\langle \theta, Z\rangle) \bigg\}
    + \frac{\lambda}{2} \| \theta \|^2
    \, ,
  \end{equation*}
  we have, respectively,
  \begin{align}
    \label{eq:proof-logistic-ridge-plusZ}
    \wh \theta_{\lambda, n}^Z
    &= \argmin_{\theta \in \R^d} \wh R_{\lambda, n}^Z (\theta) \\
    \wh \theta_{\lambda, n}^{-Z}
    &= \argmin_{\theta \in \R^d} \Big\{ \wh R_{\lambda, n}^Z (\theta) - \frac{1}{n+1} \langle \theta, Z\rangle \Big\}
      \, ,
    \label{eq:proof-logistic-ridge-minusZ}
  \end{align}
  where~\eqref{eq:proof-logistic-ridge-minusZ} comes from the fact that $\ell (-u) = \ell (u) - u$ for $u \in \R$.

  Now, the function $\wh R_n^Z$ is $\lambda$-strongly convex, as the sum of a convex function (recall that $\ell$ is convex since $\ell'' = \sigma (1 - \sigma) \geq 0$) and a $\lambda \| \theta \|^2 /2$ term.
  It follows from Lemma~\ref{lem:stability-min} that
  \begin{equation}
    \label{eq:proof-logistic-ridge-stability}
    R \cdot \big\| \wh \theta_{\lambda, n}^{-Z} - \wh \theta_{\lambda, n}^Z \big\|
    \leq R \cdot \frac{\| Z / (n+1) \|}{\lambda}
    \leq \frac{R^2}{\lambda (n+1)}
    \leq \frac{1}{2}
    \, ,
  \end{equation}
  where we used the assumption that $\lambda \geq 2 R^2 / (n+1)$.
  In addition, still by Lemma~\ref{lem:stability-min},
  \begin{equation}
    \label{eq:proof-logistic-ridge-stability-ps}
    0 \leq
    \langle \wh \theta_{\lambda, n}^{-Z} - \wh \theta_{\lambda, n}^Z, Z \rangle
    \leq 1/2
    \, .
  \end{equation}
  Now, since $(\log \sigma')' = \sigma'' / \sigma' = 1- 2\sigma \leq 1$,
  we have for every $u \in \R$ and $v \in [0, 1/2]$, $\log \sigma' (u+v) - \log \sigma' (u) \leq v$, namely $\sigma' (u+v) \leq e^v \sigma' (u) \leq e^{1/2} \cdot \sigma' (u)$.
  Hence, $\sigma (u + v) \leq e^{1/2} \cdot \sigma' (u) \cdot v$ for every $u \in \R$ and $v \in [0, 1/2]$.
  By~\eqref{eq:proof-logistic-ridge-stability-ps}, applying this inequality to $u = \langle \wh \theta_{\lambda, n}^Z, Z\rangle$ and $v = \langle \wh \theta_{\lambda, n}^{-Z} - \wh \theta_{\lambda, n}^Z, Z \rangle$ yields:
  \begin{equation}
    \label{eq:proof-logistic-ridge-stability-sigmoid}
    \sigma \big( \langle \wh \theta_{\lambda, n}^{-Z}, Z\rangle \big) - \sigma \big( \langle \wh \theta_{\lambda, n}^Z, Z\rangle \big)
    \leq e^{1/2} \cdot \sigma' \big( \langle \wh \theta_{\lambda, n}^Z, Z\rangle \big) \cdot \langle \wh \theta_{\lambda, n}^{-Z} - \wh \theta_{\lambda, n}^Z, Z \rangle
    \, .
  \end{equation}

  Let us now consider the function $\wh R_{\lambda, n}^Z$; its third derivative can be controlled in terms of its Hessian, as shown by~\cite{bach2010logistic}.
  Fix $\theta, \beta \in \R^d$, and define the function $g (t) = \wh R_{\lambda, n}^Z (\theta + t \beta)$ for $t \in \R$.
  We have respectively, denoting $\theta_t = \theta + t \beta$,
  \begin{align}
    g'' (t)
    &= \langle \nabla^2 \wh R_{\lambda, n}^Z (\theta_t) \beta, \beta \rangle 
      = \frac{1}{n+1} \left\{ \sum_{i=1}^n \sigma' (\langle \theta_t, Z_i\rangle) \langle \beta, Z_i \rangle^2 + \sigma' (\langle \theta_t, Z\rangle) \langle \beta, Z \rangle^2 \right\} + \lambda \| \beta \|^2 \nonumber \\
    g'''(t)
    &= \nabla^3 \wh R_{\lambda, n}^Z (\theta_t) [\beta, \beta, \beta]
    = \frac{1}{n+1} \left\{ \sum_{i=1}^n \sigma'' (\langle \theta_t, Z_i\rangle) \langle \beta, Z_i \rangle^3 + \sigma'' (\langle \theta_t, Z\rangle) \langle \beta, Z \rangle^3 \right\} \nonumber
  \end{align}
  Now, since $|\sigma''| = | \sigma (1 - \sigma) (1 - 2 \sigma)| \leq \sigma (1 - \sigma) = \sigma'$ (as $0\leq \sigma \leq 1$), and since by the Cauchy-Schwarz inequality $|\langle \beta, Z_i\rangle| \leq R \| \beta \|$ ($1 \leq i \leq n$) and $|\langle \beta, Z \rangle| \leq R \| \beta \|$, we have
  \begin{align}
    \label{eq:proof-logistic-ridge-self-conc}
    | g'''(t) |
    &= \frac{1}{n+1} \left\{ \sum_{i=1}^n \left| \sigma'' (\langle \theta_t, Z_i\rangle) \langle \beta, Z_i \rangle^3 \right| + \left| \sigma'' (\langle \theta_t, Z\rangle) \langle \beta, Z \rangle^3 \right| \right\}  \nonumber \\
    &\leq R \| \theta \| \cdot \frac{1}{n+1} \left\{ \sum_{i=1}^n \sigma' (\langle \theta_t, Z_i\rangle) \langle \beta, Z_i \rangle^2 + \sigma' (\langle \theta_t, Z\rangle) \langle \beta, Z \rangle^2 \right\}
      \leq R \| \beta \| \cdot g'' (t)
      \, .
  \end{align}
  The property~\eqref{eq:proof-logistic-ridge-self-conc} is the pseudo-self-concordance condition introduced by~\cite{bach2010logistic}; in particular, by Proposition~1 therein, we have for every $\theta, \beta \in \R^d$:
  \begin{equation}
    \label{eq:proof-logistic-ridge-lower-hessian}
    \nabla^2 \wh R_{\lambda, n}^Z (\theta + \beta)
    \mgeq e^{- R \| \beta \|} \cdot \nabla^2 \wh R_{\lambda, n}^Z (\theta)
    \, .
  \end{equation}
  It follows from~\eqref{eq:proof-logistic-ridge-lower-hessian} (letting $\theta = \wh \theta_{\lambda, n}^Z$ and $\beta = \theta' - \wh \theta_{\lambda, n}^Z$) that $\wh R_{\lambda, n}^Z$ is $e^{-(1/2+\eps)} \nabla^2 \wh R_{\lambda, n}^Z (\wh\theta_{\lambda, n}^Z)$-strongly convex on the open convex ball $\Omega_{\eps} = \{ \theta ' \in \R^d : R \| \theta' - \wh \theta_{\lambda, n}^Z \| < 1/2 + \eps \}$ for every $\eps > 0$.
  In addition, the inequality~\eqref{eq:proof-logistic-ridge-stability} shows that the function $\wh R_{\lambda, n}^Z (\theta) - \langle \theta, Z\rangle / (n+1)$ reaches its minimum $\wh \theta_{\lambda, n}^{-Z}$ on $\Omega_\eps$,
  so that by Lemma~\ref{lem:stability-min}, 
  \begin{equation*}
    \langle \wh \theta_{\lambda, n}^{-Z} - \wh \theta_{\lambda, n}^Z, Z / (n+1) \rangle
    \leq e^{1/2 + \eps} \left\| \frac{Z}{n+1} \right\|_{\nabla^2 \wh R_{\lambda, n}^Z (\wh \theta_{\lambda, n}^Z)^{-1}}^2
    \, .
  \end{equation*}
  Taking $\eps \to 0$ in the above bound and multiplying by $n+1$, we obtain:
  \begin{equation}
    \label{eq:proof-logistic-ridge-stab}
    \langle \wh \theta_{\lambda, n}^{-Z} - \wh \theta_{\lambda, n}^Z, Z \rangle
    \leq \frac{e^{1/2}}{n + 1} \cdot \langle \nabla^2 \wh R_{\lambda, n}^Z (\wh \theta_{\lambda, n}^Z)^{-1} Z, Z \rangle
    \, ,
  \end{equation}
  so that by combining inequalities~\eqref{eq:proof-logistic-ridge-stability-sigmoid} and~\eqref{eq:proof-logistic-ridge-stab},
  \begin{equation}
    \label{eq:proof-logistic-ridge-sigmoid-2}
    \sigma \big( \langle \wh \theta_{\lambda, n}^{-Z}, Z\rangle \big) - \sigma \big( \langle \wh \theta_{\lambda, n}^Z, Z\rangle \big)
    \leq \frac{e}{n+1} \cdot \sigma' \big( \langle \wh \theta_{\lambda, n}^Z, Z\rangle \big) \cdot \langle \nabla^2 \wh R_{\lambda, n}^Z (\wh \theta_{\lambda, n}^Z)^{-1} Z, Z \rangle
    \, .
  \end{equation}

  It thus remains to control the expectation of the right-hand side of~\eqref{eq:proof-logistic-ridge-sigmoid-2}.
  By exchangeability of $(Z_1, \dots, Z_n, Z)$ (and since $\wh R_{\lambda, n}^Z, \wh \theta_{\lambda, n}^Z$ are unchanged after permutation of $Z_i$ and $Z$), we have:
  \begin{align}    
    &\E \big[ \sigma' \big( \langle \wh \theta_{\lambda, n}^Z, Z\rangle \big) \cdot \langle \nabla^2 \wh R_{\lambda, n}^Z (\wh \theta_{\lambda, n}^Z)^{-1} Z, Z \rangle \big] \nonumber \\
    &= \frac{1}{n+1} \E \bigg[ \sum_{i=1}^n \sigma' \big( \langle \wh \theta_{\lambda, n}^Z, Z_i\rangle \big) \cdot \langle \nabla^2 \wh R_{\lambda, n}^Z (\wh \theta_{\lambda, n}^Z)^{-1} Z_i, Z_i \rangle + \sigma' \big( \langle \wh \theta_{\lambda, n}^Z, Z\rangle \big) \cdot \langle \nabla^2 \wh R_{\lambda, n}^Z (\wh \theta_{\lambda, n}^Z)^{-1} Z, Z \rangle \bigg] \nonumber \\
    &= \E \bigg[ \tr \bigg\{ \nabla^2 \wh R_{\lambda, n}^Z (\wh \theta_{\lambda, n}^Z)^{-1} \cdot \frac{1}{n+1} \bigg( \sum_{i=1}^n \sigma' \big( \langle \wh \theta_{\lambda, n}^Z, Z\rangle \big) Z_i Z_i^\top + \sigma' \big( \langle \wh \theta_{\lambda, n}^Z, Z\rangle \big) Z Z^\top \bigg) \bigg\} \bigg] \nonumber \\
    &= \E \left[ \tr \Big\{ \big[ \nabla^2 \wh R_n^Z (\wh \theta_{\lambda, n}^Z) + \lambda I_d \big]^{-1} \nabla^2 \wh R_{n}^Z (\wh \theta_{\lambda, n}^Z) \Big\} \right]
      \label{eq:proof-logistic-ridge-exchangeable}
      \, ;
  \end{align}
  in~\eqref{eq:proof-logistic-ridge-exchangeable}, we defined 
  \begin{equation*}
    \wh R_{n}^Z (\theta)
    = \wh R_{\lambda, n}^Z (\theta) - \frac{\lambda}{2} \| \theta \|^2
    = \frac{1}{n+1} \bigg\{ \sum_{i=1}^n \ell (\langle \theta, Z_i\rangle) + \ell (\langle \theta, Z\rangle) \bigg\}
    \, ,
  \end{equation*}
  whose Hessian writes
  \begin{equation*}
    \nabla^2 \wh R_{n}^Z (\theta)
    = \frac{1}{n+1} \bigg\{ \sum_{i=1}^n \sigma' (\langle \theta, Z_i\rangle) Z_i Z_i^\top + \sigma' (\langle \theta, Z\rangle) Z Z^\top \bigg\}
    \, .
  \end{equation*}
  Finally, by concavity of the map $A \mapsto \tr [ (A + \lambda \id)^{-1} A ]$ on positive matrices (as in the proof of Theorem~\ref{thm:smp-linear-ridge}), denoting $\wt H_{\lambda, n} := \E [ \nabla^2 \wh R_n^Z (\wh \theta_{\lambda, n}^Z) ] = \E [ \nabla^2 \wh R_{n+1} (\wh \theta_{\lambda, n+1}) ]$ we have
  \begin{align}
    \label{eq:proof-logistic-ridge-dflambda}
    \E \left[ \tr \Big\{ \big[ \nabla^2 \wh R_n^Z (\wh \theta_{\lambda, n}^Z) + \lambda I_d \big]^{-1} \nabla^2 \wh R_{n}^Z (\wh \theta_{\lambda, n}^Z) \Big\} \right]
    \leq \tr \big\{ [ \wt H_{\lambda, n} + \lambda \id ]^{-1} \wt H_{\lambda, n} \big\}
    = \dflambda{\wt H_{\lambda, n}}
    \, .
  \end{align}
  Combining inequalities~\eqref{eq:proof-logistic-ridge-2},~\eqref{eq:proof-logistic-ridge-sigmoid-2},~\eqref{eq:proof-logistic-ridge-exchangeable} and~\eqref{eq:proof-logistic-ridge-dflambda}, we conclude that
  \begin{equation}
    \label{eq:eq:proof-logistic-ridge-final}
    \E \big[ R (\wt f_{\lambda, n}) \big]
    - \inf_{\theta \in \R^d} \Big\{ R (f_\theta) + \frac{\lambda}{2} \| \theta \|^2 \Big\}
    \leq e \cdot \frac{\dflambda{\wt H_{\lambda, n}}}{n+1}
    \, .
  \end{equation}
  Finally, the bound~\eqref{eq:excessrisk-logistic-ridge-smp} is obtained by noting that, by exchangeability and since $\sigma' = \sigma (1 - \sigma) \leq 1/4$ and $Z_1 Z_1^\top = X_1 X_1^\top$,
  \begin{equation*}
    \wt H_{\lambda, n+1}
    = \E \big[ \sigma' (\langle \wh \theta_{\lambda, n+1}, Z_1\rangle) Z_1 Z_1^\top \big]
    \leq \E \big[ X_1 X_1^\top \big]/ 4
    = \Sigma / 4
    \, ,
  \end{equation*}
  so that $\dflambda{\wt H_{\lambda, n}} \leq \dflambda{\Sigma/4} = \df{4\lambda}{\Sigma}$.
\end{proof}

\begin{lemma}[Stability]
  \label{lem:stability-min}
  Let $\Omega$ be a nonempty open convex subset of $\R^d$, and $F : \Omega \to \R$ a differentiable function.
  Assume that $F$ is $\Sigma$-\emph{strongly convex} on $\Omega$ (where $\Sigma$ is a $d \times d$ symmetric positive matrix),
  in the sense that, for every $x, x' \in \Omega$, 
  \begin{equation}
    \label{eq:def-strongly-convex-Sigma}
    F (x') \geq F (x) + \langle \nabla F (x), x'-x\rangle + \frac{1}{2} \| x' - x \|_\Sigma^2
    \, .
  \end{equation}
  Assume that $F$ reaches its minimum at $x^* \in \Omega$.
  Let $g \in \R^d$, and assume that the function $x \mapsto F (x) - \langle g, x\rangle$ reaches its minimum at some $\wt x \in \Omega$.
  Then, 
  \begin{equation}
    \label{eq:stability-strongly-convex}
    \| \wt x - x^* \|_\Sigma \leq \| g \|_{\Sigma^{-1}} \, , \qquad
    \langle g, \wt x - x^*\rangle
    \leq \| g \|_{\Sigma^{-1}}^2
    \, .
  \end{equation}
\end{lemma}

\begin{proof}
  First, since $\wt x \in \Omega$ minimizes the function $x \mapsto F (x) - \langle g, x\rangle$, we have $0 = \nabla F (\wt x) - g$.
  This implies
  \begin{equation}
    \label{eq:proof-stability-strongly-1}
    \langle \nabla F (\wt x), \wt x - x^* \rangle
    = \langle g, x\rangle
    \, .
  \end{equation}
  Now, by substituting $x'$ and $x$ in inequality~\eqref{eq:def-strongly-convex-Sigma} and adding the resulting inequality to~\eqref{eq:def-strongly-convex-Sigma}, we obtain for every $x, x' \in \Omega$,
  \begin{equation*}
    \langle \nabla F (x') - \nabla F (x), x' - x \rangle
    \geq \| x' - x \|_\Sigma^2
    \, .
  \end{equation*}
  Setting $x' = \wt x$ and $x = x^*$, and using that $\nabla F (x^*) = 0$ (since $x^* \in \Omega$ minimizes $F$), we obtain $\langle \nabla F (\wt x), \wt x - x^* \rangle \geq \| \wt x - x^* \|^2_\Sigma$.
  On the other hand, the Cauchy-Schwarz inequality implies that
  \begin{equation}
    \label{eq:proof-stability-strongly-cs}
    \langle g, \wt x - x^*\rangle \leq \| g \|_{\Sigma^{-1}} \cdot \| \wt x - x^* \|_\Sigma
    \, .
  \end{equation}
  Plugging the previous inequalities in~\eqref{eq:proof-stability-strongly-1} yields 
  $\| x' - x \|_\Sigma^2 \leq \| g \|_{\Sigma^{-1}} \cdot \| \wt x - x^* \|_\Sigma$, hence $\| x' - x \|_\Sigma \leq \| g \|_{\Sigma^{-1}}$;
  the inequality $\langle g, \wt x - x^*\rangle \leq \| g \|_{\Sigma^{-1}}^2$ then follows by~\eqref{eq:proof-stability-strongly-cs}.
\end{proof}

\end{document}